\documentclass[12pt,reqno,twoside,fleqn]{amsart}
\usepackage[mathcal]{euler}
\usepackage{palatino}
\usepackage{amsmath,amssymb,amsthm,graphics}%
\usepackage{tikz-cd}
\usetikzlibrary{matrix,arrows,decorations.pathmorphing}

\numberwithin{equation}{section}
\makeatletter
\renewcommand{\section}{\@startsection {section}{1}{\z@}%
                                   {-3.5ex \@plus -1ex \@minus -.2ex}%
                                   {.5\linespacing}%
                                   {\normalfont\scshape\centering}}
\makeatother

\newtheorem{thm}{Theorem}[section]
\newtheorem{lem}[thm]{Lemma}
\newtheorem{cor}[thm]{Corollary}

\theoremstyle{definition}
\newtheorem{definition}{Definition}[section]

\theoremstyle{remark}
\newtheorem*{remark}{Remark}

\def\beq#1\eeq{\begin{equation}#1\end{equation}}
\makeatletter
\@ifpackageloaded{euler}{
 \newcommand{\onto}{\to\mkern-14mu\to}
 \def\hbar{{\mathchar'26\mkern-6.5muh}}
 \newcommand{\Wedge}{{\textstyle\bigwedge}}
 
}{
 \DeclareMathSymbol{\onto}{\mathrel}{AMSa}{"10}
 \renewcommand{\hbar}{{\mathchar'26\mkern-9muh}}
 \newcommand{\Wedge}{{\textstyle\bigwedge}}
 
}
\makeatother
\newcommand{\C}{\mathcal{C}}
\newcommand{\Ci}{\C^\infty}
\newcommand{\co}{\mathbb{C}}
\newcommand{\cs}{\mbox{\upshape C}\ensuremath{{}^*}}

\newcommand{\R}{\mathbb{R}}
\newcommand{\Z}{\mathbb{Z}} 
\newcommand{\into}{\hookrightarrow}
\DeclareMathOperator{\id}{id}
\DeclareMathOperator{\rk}{rk}
\newcommand{\inner}{\mathbin{\raise1.5pt\hbox{$\lrcorner$}}}
\newcommand{\K}{\mathcal{K}}
\newcommand{\G}{\mathcal{G}}
\DeclareMathOperator{\pr}{pr}
\newcommand{\T}{\mathbb{T}}
\DeclareMathOperator{\inv}{inv}
\newcommand{\E}{\mathcal{E}}
\newcommand{\D}{\mathcal{D}}
\newcommand{\F}{\mathcal{F}}
\renewcommand{\P}{\mathcal{P}}
\newcommand{\Q}{\mathcal{Q}}
\newcommand{\Rp}{\mathcal{R}}
\newcommand{\Lie}{\mathcal{L}}
\newcommand{\Li}{\mathcal{L}}
\newcommand{\s}{\mathsf{s}}
\renewcommand{\t}{\mathsf{t}}
\newcommand{\m}{\mathsf{m}}
\newcommand{\unit}{\mathsf{1}}
\newcommand{\abs}[1]{\lvert#1\rvert}
\newcommand{\Kahler}{K\"ah\-ler} 
\newcommand{\Podles}{Podle\`{s}}

\newcommand{\Wedgem}{\Wedge^{\mathrm{max}}}
\newcommand{\X}{\mathcal{X}}
\DeclareMathOperator{\Pair}{Pair}
\newcommand{\BS}{\text{B-S}}
\newcommand{\Hei}{\mathbb{H}}
\newcommand{\Hi}{\mathcal H}
\newcommand{\N}{\mathbb N}
\DeclareMathOperator{\Exp}{Exp}
\newcommand{\jac}{\mathcal{S}}
\newcommand{\U}{\mathrm{U}}
\newcommand{\SU}{\mathrm{SU}}
\DeclareMathOperator{\PPr}{Pr}
\renewcommand{\Pr}{\PPr}
\newcommand{\EE}{\mathbb{E}}
\newcommand{\su}{\mathfrak{su}}
\DeclareMathOperator{\Ad}{Ad}
\newcommand{\Omegan}{\Omega^1_{\mathrm{normal}}}
\DeclareMathOperator{\EA}{\mathsf{ex}}
\newcommand{\mon}{\mathcal{N}}
\newcommand{\Chi}{X}
\newcommand{\pro}{q}
\DeclareMathOperator{\Res}{\mathcal{R}}
\DeclareMathOperator{\ev}{ev}
\newcommand{\period}{\ .\hspace*{-.56em}}
\newcommand{\norm}[1]{\lVert#1\rVert}
\newcommand{\Norm}[1]{\left\|#1\right\|}
\newcommand{\naturalto}{\mathrel{\dot\to}}

\title{Quantization of Planck's Constant}
\author{Eli Hawkins}

\begin{document}
\maketitle
\begin{center}
\vspace{-4ex}
\emph{\small Department of Mathematics}\\
\emph{\small The University of York, United Kingdom}\\
{\small mrmuon@mac.com}\\
\end{center}

\begin{abstract}
This paper is about the role of Planck's constant, $\hbar$,  in the geometric quantization of Poisson manifolds using symplectic groupoids. In order to construct a strict deformation quantization of a given Poisson manifold,  one can use all possible rescalings of the Poisson structure, which can be combined into a single  ``Heisenberg-Poisson'' manifold. The new coordinate on this manifold is identified with $\hbar$. I present an explicit construction for a symplectic groupoid integrating a Heisenberg-Poisson manifold and discuss its geometric quantization. I show that in cases where $\hbar$ cannot take arbitrary values, this is enforced by Bohr-Sommerfeld conditions in  geometric quantization.

A Heisenberg-Poisson manifold is defined by linearly rescaling the Poisson structure, so I also discuss nonlinear variations and give an example of quantization of a nonintegrable Poisson manifold using a presymplectic groupoid. 

In appendices, I construct  symplectic groupoids integrating a more general class of Heisenberg-Poisson manifolds constructed from  Jacobi manifolds and discuss the parabolic tangent groupoid.
\end{abstract}

\tableofcontents

\section{Introduction}
The term ``quantization'' is used in many ways in mathematics. Geometric quantization is a method of constructing a Hilbert space based upon a symplectic manifold, $(M,\omega)$. In practice, this has mainly been used to construct unitary group representations, but it was originally intended as a way of constructing the state space of a quantum-mechanical model from the corresponding classical phase space (the symplectic manifold). From this perspective, operators on the Hilbert space should be quantum observables and correspond to functions on phase space. This correspondence is the ``classical limit'', realized in a limit of diverging symplectic form.

This correspondence is made precise with the notion of \emph{strict deformation quantization}. Among other structures, this involves a continuous field of \cs-algebras over a subset $I\subseteq\R$ of the real line. The coordinate on $\R$ is usually denoted $\hbar$ and thought of as Planck's constant. The algebra at $\hbar=0$ is the algebra $\C_0(M)$ of continuous functions vanishing at $\infty$. 

For $(M,\omega)$, the algebra at $\hbar\neq 0$ is the algebra $\K(\Hi_\hbar)$ of compact operators  on the Hilbert space constructed by geometric quantization of $(M,\hbar^{-1}\omega)$. 

The set $I\subseteq\R$ can be very different in different examples. The standard geometric quantization construction requires a line bundle with curvature equal to the symplectic form. In this construction of a deformation quantization, that means that for any $\hbar\neq0\in I$, the cohomology class of $\frac\omega{2\pi\hbar}$ must be integral.

This condition can be highly restrictive or  vacuous. For example, the cohomology class of a symplectic form $\omega\in\Omega^2(S^2)$ on a sphere is nontrivial. By rescaling, we can suppose that $[\frac{\omega}{2\pi}]\in H^2(S^2)$ is the generator of $H^2(S^2;\Z)$, in which case
\[
I = \{0,k^{-1}\mid k\in\N\} .
\]
On the other hand, for $\R^2$, the cohomology is trivial, so $[\frac\omega{2\pi\hbar}]=0$ is trivially integral. This leads to
\[
I = \R_{\geq0} \text{ or } \R
\]
(depending upon the choice of polarization).

In using standard geometric quantization to construct a strict deformation quantization of a symplectic manifold, we must put in the choice of $I$ by hand, although  if we insist upon $I=\R_{\geq0}$ for $S^2$, then the construction cannot be carried out. 

Quantization is not limited to symplectic manifolds. The idea of strict deformation quantization applies to any Poisson manifold, but Hilbert spaces are less useful in general. The algebra of compact operators on a Hilbert space is not general enough to quantize most Poisson manifolds.

In \cite{haw10}, I proposed a generalization of geometric quantization to Poisson manifolds. The idea is to construct a \cs-algebra from a Poisson manifold, $M$, using several structures, including a symplectic groupoid. In particular, if $M$ is symplectic and the symplectic groupoid is the pair groupoid, $\Pair M = M\times M$, then this construction can give the algebra $\K(\Hi)$, where $\Hi$ is a Hilbert space constructed from $M$ by standard geometric quantization.

Like standard geometric quantization, my proposal is a work in progress. I do not have a completely general definition for constructing the \cs-algebra. 

If $M$ is a manifold with Poisson structure $\Pi\in\X^2(M)$, then to construct a strict deformation quantization of $(M,\Pi)$, we first need a family of \cs-algebras. The simplest choice is to construct these by rescaling the Poisson structure to $\hbar\Pi$ and constructing an algebra from that. In particular, when $\hbar=0$, this just gives $\C_0(M)$. As with standard geometric quantization, my construction requires certain integrality conditions to be satisfied, so these algebras are typically only defined for a discrete subset of nonzero $\hbar$'s.

The next step is to assemble these algebras into a continuous field. This means identifying an algebra of ``continuous sections'' which is a \cs-subalgebra of the direct product of the collection of algebras. From the perspective of noncommutative geometry, this means assigning a (noncommutative) topology to a union of (noncommutative) topological spaces. 

Suppose for a moment that there are no nontrivial integrality conditions, and a \cs-algebra can be constructed from $(M,\hbar\Pi)$ for all $\hbar\in\R$. In that case, we want to construct a union over $\R$ of noncommutative spaces. The noncommutative space at $\hbar\in\R$ corresponds to $(M,\hbar\Pi)$, so the noncommutative union should correspond to the union of these Poisson manifolds. It should be constructed from $M\times\R$ with Poisson structure $\hbar\Pi$, where $\hbar$  now means the coordinate on the $\R$ factor.

I will show in this paper  that this still works when there are nontrivial integrality conditions.  The Bohr-Sommerfeld quantization condition automatically picks out the values of $\hbar$ satisfying the necessary integrality conditions. 
This means that the set $I\subseteq\R$ of allowed $\hbar$-values does not have to be chosen by hand. It emerges automatically when constructing a \cs-algebra from $(M\times\R,\hbar\Pi)$.

\subsection{Background}
I will assume that the reader is familiar with symplectic groupoids, but I shall quickly review some notations, terminology, constructions and results.

\begin{definition}
\label{Groupoid def}
A \emph{Lie groupoid} \cite{mac3,m-m} consists of a not necessarily Hausdorff manifold, $\G$, a Hausdorff manifold, $M$ ($\G$ is called a Lie groupoid \emph{over} $M$), and  the following maps: the \emph{unit} is $\unit:M\into\G$, the \emph{inverse} is $\inv:\G\to\G$, the \emph{source} is $\s:\G\to M$,  the \emph{target} is $\t:\G\to M$, and the multiplication is $\m :\G_2\to\G$, where $\G_2 := \G \mathbin{{}_\s\times_\t} \G$ is the manifold of composable pairs. (I regard a groupoid element as an arrow from right to left, in analogy with the standard notation for linear operators.) The Cartesian projections are $\pr_1,\pr_2:\G_2\to\G$. More generally, $\G_n\subset \G^n$ is the manifold of composable $n$-tuples. The target map is required to be a submersion with Hausdorff fibers. The following diagrams are commutative:
\begin{equation}
\label{associativity}
\begin{tikzcd}
\G_3 \arrow{r}{\m\times\id_\G\;} \arrow{d}[swap]{\id_\G\times\m} & \G_2 \arrow{d}{\m} \\
\G_2 \arrow{r}{\m} & \G
\end{tikzcd}
\end{equation}
\begin{equation}
\label{s and t}
\text{(a)}~
\begin{tikzcd}
\G_2 \arrow{r}{\pr_2} \arrow{d}{\m} & \G \arrow{d}{\s} \\
\G \arrow{r}{\s} & M
\end{tikzcd}
\qquad\text{(b)}~
\begin{tikzcd}
\G_2 \arrow{r}{\pr_1} \arrow{d}{\m} & \G \arrow{d}{\t} \\
\G \arrow{r}{\t} & M\period
\end{tikzcd}
\end{equation}
There exists \emph{one} map completing the three commutative diagrams
\begin{equation}
\label{right unit}
\begin{tikzcd}
\G \arrow[dashed]{r}{} \arrow{dr}[swap]{\id_\G} & \G_2\arrow{d}{\pr_1}\\
& \G
\end{tikzcd}
\qquad
\begin{tikzcd}
\G \arrow[dashed]{r}{} \arrow{dr}[swap]{\id_\G} & \G_2\arrow{d}{\m}\\
& \G
\end{tikzcd}
\qquad
\begin{tikzcd}
\G \arrow[dashed]{r}{} \arrow{d}{\s} & \G_2 \arrow{d}{\pr_2} \\
M \arrow{r}{\unit} & \G
\end{tikzcd}
\end{equation}
and another completing the diagrams
\begin{equation}
\label{left unit}
\begin{tikzcd}
\G \arrow[dashed]{r}{} \arrow{d}{\t} & \G_2 \arrow{d}{\pr_1} \\
M \arrow{r}{\unit} & \G
\end{tikzcd}
\qquad
\begin{tikzcd}
\G \arrow[dashed]{r}{} \arrow{dr}[swap]{\id_\G} & \G_2\arrow{d}{\m}\\
& \G
\end{tikzcd}
\qquad
\begin{tikzcd}
\G \arrow[dashed]{r}{} \arrow{dr}[swap]{\id_\G} & \G_2\arrow{d}{\pr_2}\\
& \G \period
\end{tikzcd}
\end{equation}
These commute:
\begin{equation}
\label{s inv t}
\text{(a)}~
\begin{tikzcd}
\G\arrow{r}{\inv}\arrow{dr}[swap]{\t} & \G\arrow{d}{\s}\\
& M
\end{tikzcd}
\qquad\text{(b)}~
\begin{tikzcd}
\G\arrow{r}{\inv}\arrow{dr}[swap]{\s} & \G\arrow{d}{\t}\\
& M\period
\end{tikzcd}
\end{equation}
There exists another map completing the diagrams
\begin{equation}
\label{left inverse}
\begin{tikzcd}
\G \arrow[dashed]{r}{} \arrow{dr}[swap]{\inv} &\G_2\arrow{d}{\pr_1}\\
& \G
\end{tikzcd}
\qquad
\begin{tikzcd}
\G \arrow[dashed]{r}{} \arrow{d}{\s} & \G_2 \arrow{d}{\m} \\
M \arrow{r}{\unit} & \G
\end{tikzcd}
\qquad
\begin{tikzcd}
\G \arrow[dashed]{r}{} \arrow{dr}[swap]{\id_\G} &\G_2\arrow{d}{\pr_2}\\
& \G
\end{tikzcd}
\end{equation}
and one more completing the diagrams
\begin{equation}
\label{right inverse}
\begin{tikzcd}
\G \arrow[dashed]{r}{} \arrow{dr}[swap]{\id_\G} &\G_2\arrow{d}{\pr_1}\\
& \G
\end{tikzcd}
\qquad
\begin{tikzcd}
\G \arrow[dashed]{r}{} \arrow{d}{\t} & \G_2 \arrow{d}{\m} \\
M \arrow{r}{\unit} & \G
\end{tikzcd}
\qquad
\begin{tikzcd}
\G \arrow[dashed]{r}{} \arrow{dr}[swap]{\inv} &\G_2\arrow{d}{\pr_2}\\
& \G \period
\end{tikzcd}
\end{equation}
\end{definition}
\begin{remark}
This expression of the definition in terms of commutative diagrams will be the most useful version for this paper.
In practice, it is usually not convenient to construct $\G_2$ as a subset of $\G\times\G$. Instead, it is treated as a space with structure maps satisfying the definition of a pullback.
\end{remark}

The Lie algebroid of $\G$ is (as a bundle) the normal bundle, $\nu_M\G$, to the unit submanifold (image of $\unit$). The anchor map of a Lie algebroid is denoted $\#$.

\begin{definition}
The multiplicative coboundary (in degree $1$), $\partial^* : \Omega^*(\G) \to \Omega^*(\G_2)$, is $\partial^* = \mathord{\pr}_1^*-\m^*+\mathord{\pr}_2^*$. A \emph{symplectic groupoid}, $(\Sigma,\omega)$ consists of a groupoid $\Sigma$ with a symplectic form $\omega\in\Omega^2(\Sigma)$ that is \emph{multiplicative} in the sense that $\partial^*\omega=0$. A symplectic groupoid $(\Sigma,\omega)$ \emph{integrates} a Poisson manifold $(M,\Pi)$ if $\Sigma$ is an $\s$-connected groupoid over $M$ and $\t:\Sigma\to M$ is a Poisson map (i.e., $\t^*$ intertwines Poisson brackets). 
\end{definition}
\begin{remark}
Some authors instead choose $\s$ to be a Poisson map. I believe that my convention is the most appropriate to geometric quantization. The advantage is apparent from the case when $\Sigma$ is the pair groupoid of a symplectic manifold.
\end{remark}

\begin{definition}
\label{Exponential}
If $\G$ is a Lie groupoid over $M$ with Lie algebroid $A$, then an \emph{exponential map} $\Exp_\G:A\to\G$ is a smooth map such that:
\begin{itemize}
\item
For $x\in M$, $\Exp_\G0_x = \unit(x)$.
\item
The linearization of $\Exp_\G$ about the $0$ section is equal to the canonical identification of $A$ with the normal bundle  $\nu_M\G$.
\item
There exists a neighborhood $U\subset A$ of the $0$ section such that the restriction of $\Exp_\G$ to $U$ is a diffeomorphism to its image, $\Exp_\G U$.
\end{itemize}
\end{definition}
\begin{remark}
Exponential maps always exist. For example, from a connection on any Lie algebroid, Landsman constructs ``left exponential'' and ``Weyl exponential'' maps with these properties \cite[Thm.~3.10.6]{lan2}. This is a very weak definition, chosen to capture the properties of Landsman's examples that I will actually need here.
\end{remark}

Next, I summarize my main definitions and constructions from \cite{haw10}.

\begin{definition}
A \emph{prequantization, $(L,\nabla,\sigma)$, of a symplectic groupoid} $(\Sigma,\omega)$  consists of a Hermitian line bundle $L\to\Sigma$ with connection $\nabla$ and a section $\sigma\in\Gamma(\Sigma_2,\pr_1^*L^*\otimes\m^*L\otimes\pr_2^*L^*)$  such that:
\begin{enumerate}
\item
$\nabla^2=-i\omega$;
\item
$\sigma$ is a (multiplicative) cocycle and has norm $1$ at every point; 
\item
$\sigma$ is covariantly constant.
\end{enumerate}
\end{definition}
The principal $\T$-bundle associated to a prequantization is itself a groupoid, with multiplication determined by the cocycle, $\sigma$.

Recall that applying the tangent functor to the structure maps of a groupoid, $\G$, defines another groupoid, $T\G$.
\begin{definition}
A \emph{groupoid polarization} of $\G$ is a subbundle $\P\subset T_\co\G$ (a \emph{tangent distribution}) with these three properties:
\begin{description}
\item[Involutive]
The space of smooth sections of $\P$ is closed under the Lie bracket.
\item[Hermitian]
$T\inv(\P)=\bar\P$.
\item[Multiplicative]
For all $(\gamma,\eta)\in \G_2$, 
\[
\P_{\gamma\eta} = \{T\m(X,Y) \mid X\in \P_\gamma,\ Y\in \P_\eta,\ T\s(X)=T\t(Y)\}\ .
\]
\end{description}
\end{definition}
\begin{definition}
A \emph{symplectic groupoid polarization} of $(\Sigma,\omega)$  is a groupoid polarization $\P\subset T_\co\Sigma$ that is also Lagrangian at every point.
\end{definition}

A polarization determines several other (typically singular) tangent distributions, which we will need.
\begin{definition}
\label{Distributions}
For any Lie groupoid, $\G$, define the bundles $T^\t\G := \ker T\t$ and $T^\s\G := \ker T\s$.

Given a polarization, $\P$, of a groupoid $\G$ over $M$, the real distributions $\D,\E\subseteq T\G$ are defined by $\D^\co := \P\cap\bar\P$ and $\E^\co := \P+\bar\P$; the complex distribution $\P_0\subset T^\co M$ is defined by $\P_{0\,x} := T\t(\P_{\unit(x)})$ for any $x\in M$; the real distribution $\D_0\subseteq TM$ is defined by $\D_0^\co := \P_0\cap\bar\P_0$. Finally, $\P^\t := T^\t_\co\G\cap\P$, \emph{et cetera}.
\end{definition}

Given a polarization $\P$ of a  groupoid $\G$, define the line bundle 
\beq
\label{form bundle}
\Omega_\P := 
\Wedgem\left(T^\t_\co\G/\P^\t\oplus T^\s_\co\G / \P^\s\right)^*
\eeq
provided that $\P^\t$ is locally trivial. The \emph{half-form bundle} $\Omega^{1/2}_\P$ is defined to be some square root of $\Omega_\P$; that is, a line bundle equipped with an isomorphism
\[
\Omega^{1/2}_\P \otimes \Omega^{1/2}_\P \cong \Omega_\P \ .
\]
 The Bott connection for $\P$ determines a flat $\P$-connection on $\Omega^{1/2}_\P$. It is thus meaningful to speak of ``polarized'' (that is, $\P$-constant) sections of $\Omega^{1/2}_\P$.  

Given a prequantization $(L,\nabla,\sigma)$ and a polarization $\P$ of a symplectic groupoid $\Sigma$, 
the flat $\P$-connection on $\Omega^{1/2}_\P$ and the connection on $L$ combine to determine a flat $\P$-connection on $L\otimes \Omega^{1/2}_\P$. This determines a sheaf of $\P$-polarized sections of this line bundle. In the best case, there exist enough globally polarized sections to construct a convolution algebra.

Suppose that $a,b\in\Gamma(\Sigma,L\otimes \Omega^{1/2}_\P)$ are globally polarized sections. I want to define the convolution product $a*b$ as another globally polarized section of the same line bundle. The first step is to pull $a$ and $b$ back to $\Sigma_2$ and multiply, but this gives
\[
\pr_1^*a\cdot\pr_2^*b \in \Gamma(\Sigma_2,\pr_1^*L\otimes\pr_2^*L\otimes\pr_1^*\Omega^{1/2}_\P\otimes\pr_2^*\Omega^{1/2}_\P)\ .
\]
The factors of $L$ are not quite right, but this is corrected by the cocycle, $\sigma$:
\beq
\label{raw twisted}
\pr_1^*a\cdot\pr_2^*b\cdot \sigma \in \Gamma(\Sigma_2,\m^*L \otimes\pr_1^*\Omega^{1/2}_\P\otimes\pr_2^*\Omega^{1/2}_\P)\ .
\eeq
To define $(a*b)(\gamma)$ for any $\gamma\in\Sigma$, this needs to be integrated over the $\m$-fiber $\m^{-1}(\gamma)$, which  is parameterized by the $\t$-fiber,  by
\begin{gather*}
\t^{-1}[\t(\gamma)] \to \m^{-1}(\gamma) \\
\eta \mapsto (\eta,\eta^{-1}\gamma)\ .
\end{gather*}
The product \eqref{raw twisted} is actually $\D$-polarized in this fiber, so I actually don't want to integrate over those irrelevant directions. Instead, we should integrate over a complete transversal to $\D$ in the $\t$-fiber. 

This requires a further correction, because the line bundle doesn't contain quite the right factor to integrate in this way. This comes from the difference between $\E^\t$ and $\D^\t$. 
\begin{definition}
\label{Epsilon factor}
Let $m := \rk\P_0 - \rk \D_0$ (and suppose that this is constant). Let $\varepsilon_\P \in \Gamma(\Gamma,\Wedgem[\E^\t/\D^\t]^*)$ be the restriction of
\[
\frac{1}{m!}\left(\frac{-\omega}{2\pi}\right)^m \ .
\]
\end{definition}
This is a nonvanishing section (hence a trivialization) of this line bundle  \cite[Thm.~5.3]{haw10}. Its square root provides precisely the needed correction.
\begin{remark}
The power of $2\pi$ is not absolutely necessary, but we will see in examples that this is the most convenient normalization.
\end{remark}

\begin{definition}
The convolution product of two globally $\P$-polarized sections $a,b\in\Gamma(\Sigma,L\otimes \Omega^{1/2}_\P)$ is given at $\gamma\in\Sigma$ by
\[
(a*b)(\gamma) = \int a(\eta) b(\eta^{-1}\gamma) \sigma(\eta,\eta^{-1}\gamma) \sqrt{\varepsilon_\P(\eta)}\ ,
\]
where the integral is over all $\eta\in\t^{-1}[\t(\gamma)]$ in a complete transversal to $\D$.

The involution is given at $\gamma$ by $a^*(\gamma)=\bar a(\gamma^{-1})$.

The \emph{twisted polarized convolution algebra} is the set of polarized sections that fall off rapidly, with this product and involution. Let $\cs_\P(\Sigma,\sigma)$ denote the maximal completion of this ${}^*$-algebra to a \cs-algebra. I will also sometimes denote this as $\cs_\P(\Sigma,L)$ if the prequantization is $(L,\nabla,1)$ or as $\cs_\P(\Sigma,\theta)$ if the prequantization is $(\co\times\Sigma,d+i\theta,1)$.
\end{definition}
The idea is that if $\Sigma$ integrates a Poisson manifold $(M,\Pi)$, then $\cs_\P(\Sigma,\sigma)$ should be a quantization of the algebra $\C_0(M)$.

There are some amendments to this construction that will be needed in this paper.

The first is the inclusion of \emph{Bohr-Sommerfeld conditions} \cite{sni}. The real part, $\D$, of the polarization is a foliation of $\Sigma$. A polarized section $a$ of $L\otimes\Omega^{1/2}_\P$ must in particular be covariantly constant along the leaves of $\D$, and if there is nontrivial holonomy around a multiply connected leaf, then $a$ must vanish there. If $\D$ has multiply connected leaves, then the holonomy is generically nontrivial and a continuous polarized section must vanish everywhere. The standard solution to this problem is to relax the requirement of continuity, or equivalently to deal with polarized sections over the \emph{Bohr-Sommerfeld subvariety}, $\Sigma_\BS$, which is defined to be the set of points of $\Sigma$ through which there is trivial holonomy around the leaves of $\D$. This approach works for quantization using symplectic groupoids \cite{wei4,haw10} when $\Sigma_\BS\subset\Sigma$ is a subgroupoid.

Another issue is particular to complex polarizations. A negatively curved line bundle on a complex manifold either has no nonzero holomorphic sections, or else the holomorphic sections grow very rapidly. In this way, the polarized sections of $L\otimes\Omega_\P^{1/2}$ may all vanish or be badly behaved over at least part of $\Sigma$. A sufficient condition to prevent this is:
\begin{definition}
\label{Positive}
A polarization is \emph{positive} \cite{sni} at $\gamma\in\Sigma$ if for any $X\in\P_\gamma$, 
\[
i\omega(X,\bar X) \geq 0\ .
\]
\end{definition}
This suggests, in analogy with the Bohr-Sommerfeld conditions, to only work with the part of $\Sigma$ over which $\P$ is positive. 

Another problem is that $\P^\t$ is not necessarily locally trivial, and hence $\Omega_\P$ may not be a bundle at all. In these cases, we need to keep in mind that the important thing is not the bundle, but the sheaf of polarized sections. It is reasonably apparent what this should mean in some cases, including those discussed in this paper.

\subsection{Strict deformation quantization}
The idea of strict deformation quantization began with Rieffel \cite{rie4}. There  are several variations on the definition, which I surveyed in \cite{haw11}. Most versions of strict deformation quantization of a Poisson manifold $(M,\Pi)$ use two main structures: a continuous field of \cs-algebras, $A$, over some $I\subseteq \R$ and a linear map $Q:\Ci_0(M)\to \Gamma(I,A)$.
Let's also write $Q_\hbar:\Ci_0(M)\to A_\hbar$ for the composition of $Q$ with the evaluation at $\hbar\in I$.

The most important axioms are that:
\begin{enumerate}
\item
$0\in I$ and is an accumulation point.
\item
$A_0 = \C_0(M)$ and $Q:\Ci_0(M)\to\C_0(M)$ is the inclusion.
\item
For any $f,g\in \Ci_0(M)$,
\[
\lim_{\hbar\to0} \Norm{\hbar^{-1}[Q_\hbar(f),Q_\hbar(g)]-iQ_\hbar(\{f,g\})} =0\ ,
\]
where $\{\;\cdot\;,\;\cdot\;\}$ is the Poisson bracket determined by $\Pi$ on $\Ci_0(M)$.
\end{enumerate}

This makes precise the idea that the Poisson algebra of smooth functions on $(M,\Pi)$ is an approximation to the noncommutative algebra $A_1$. More generally, $(M,\hbar\Pi)$ corresponds to $A_\hbar$.

In the spirit of Noncommutative Geometry, we can think of \ $\Gamma(I,A)$ as the algebra of functions on a noncommutative space over $I$. 
Each of the fibers of the projection to $I$ is a noncommutative version of $M$. The fiber at $0\in I$ is $M$ itself, while the algebra of functions on the fiber at $\hbar\in I$ is $A_\hbar$.

The continuous field, $A$, is the most important structure here. The map $Q$ simply gives a preferred way of extending each smooth function on $M$ to this larger noncommutative space. More than one choice of $Q$ will encode the same essential information. 

I would like to systematically construct strict deformation quantizations using the quantization recipe summarized in the previous section. In order to construct the continuous field, we need the algebras $A_\hbar$ and the algebra of continuous sections, $\Gamma(I,A)$. In principle, $A_\hbar$ should be constructed by quantizing $(M,\hbar\Pi)$. To get $\Gamma(I,A)$, we should quantize $M\times I$ with a Poisson structure such that the projection to $I$ is a Poisson map and the fiber over $\hbar\in I$ is $(M,\hbar\Pi)$.

Of course, $M\times I$ is not necessarily a manifold, but it is if $I=\R$, and this case motivates the definition of the Heisenberg-Poisson manifold of $(M,\Pi)$.

\subsection{Other notation}
I use the standard notations $\C$ for continuous functions, $\Ci$ for smooth (infinitely differentiable) functions, $\Gamma$ for smooth sections of a vector bundle, $\Wedge^*$ for exterior powers, $\Wedgem$ for maximal exterior power, $\Omega^*$ for differential forms, and $\X^*$ for multivector fields. The symbol $\inner$ denotes contraction of a vector into a differential form.
For any set, $A$, $\id_A:A\to A$ is the identity map.
The normal bundle to a submanifold $N\subset M$ is $\nu_N M = TM\rvert_N/TN$.

$\R$ is the set of real numbers. $\R_*$ is the nonzero real numbers. $\R_+$ is the strictly positive real numbers. $\co$ is the set of complex numbers. $\Z$ is the set of integers. $\N$ is the set of strictly positive integers. $\T\subset\co$ is the group of complex numbers of modulus $1$.

There are two notations associated with Cartesian products that I have found very useful here, but which may seem cryptic at first. If $f:A\to B$ and $g:C\to D$ are maps, then their Cartesian product is $f\times g :A\times C \to B\times D$. On the other hand, if $f:A\to B$ and $g:A\to C$, then $(f,g) : A \to B\times C, a\mapsto (f(a),g(a))$.

\begin{remark}
The symbol $\hbar$ has multiple meanings in this paper, and there is some potential for confusion. Several of the manifolds considered here are manifolds over $\R$, and $\hbar$ denotes the canonical projection to $\R$ from any of these. It can also be thought of as a coordinate, and so I will sometimes speak about subsets determined by values of $\hbar$. Finally, $\hbar$ is also an element of the algebra of functions on any of these manifolds, and after quantization this becomes the parameter, $\hbar$, used in deformation quantization. 
\end{remark}

\subsection{Outline}
In Section~\ref{Heisenberg-Poisson}, I construct a symplectic groupoid that integrates the Heisenberg-Poisson manifold of an arbitrary Poisson manifold. I begin with several structures constructed from a Poisson manifold: the Heisenberg-Poisson manifold, the cotangent Lie algebroid associated to the Poisson structure, and another Lie algebroid associated to the Poisson structure as a special case of a Jacobi structure. Comparing these constructions suggests the structure of the symplectic groupoid integrating the Heisenberg-Poisson manifold. The explicit construction requires some general results about double explosions of manifolds. 

In Section~\ref{Quantization HP}, I discuss geometric quantization using this symplectic groupoid. First, I discuss the prequantization, polarization, and half-form bundle needed in this construction. Then, I analyze what happens in the cases when there are or are not integrality conditions. I show how Bohr-Sommerfeld conditions enforce these integrality conditions.

In Section~\ref{Examples}, I look at the explicit geometric quantization of the Heisenberg-Poisson manifold of three examples: a vector space with constant Poisson structure, the sphere $S^2$ with a symplectic structure, and the dual of a Lie algebroid. In the case of the sphere, it is easy to compare several perspectives on the quantization.

In Section~\ref{Other}, I consider generalizations of  Heisenberg-Poisson manifolds, in which the Poisson structure is not simply rescaled linearly. I show that these tend to not be integrable to symplectic groupoids. In one example, I show how my geometric quantization  construction can be generalized to give the appropriate quantization.

In Appendix~\ref{Exploding contact}, I give a further generalization of the construction of the symplectic groupoid. Jacobi manifolds are a generalization of Poisson manifolds, and Heisenberg-Poisson manifolds can also be constructed from Jacobi manifolds. My construction of a symplectic groupoid integrating a Heisenberg-Poisson manifold extends to this generality.

In Appendix~\ref{Parabolic}, I describe how the double explosion construction that I use here can be used to construct the  parabolic tangent groupoid with which Erik van~Erp \cite{erp1} computed the indices of subelliptic operators on contact manifolds.

\section{Heisenberg-Poisson manifolds}
\label{Heisenberg-Poisson}
\subsection{Heisenberg and Jacobi}
Let $M$ be a  manifold with Poisson bivector $\Pi\in\X^2(M)$. This defines a Poisson bracket by $\{f,g\}:=\Pi(df,dg)$ for any $f,g\in\Ci(M)$.

There are two important Lie algebroids defined by a Poisson structure. 
\begin{definition}
The \emph{cotangent Lie algebroid} of $(M,\Pi)$ is the cotangent bundle $T^*M$ with the anchor map $\#_\Pi:T^*M\to TM$ defined by contraction with $\Pi$ and the \emph{Koszul bracket} $[\,\cdot\,,\,\cdot\,]_\Pi$ --- which is the unique Lie algebroid bracket such that for any $f,g\in\Ci(M)$,
\[
[df,dg]_\Pi = d\{f,g\} .
\]
\end{definition}

Poisson structures can be viewed as a special type of Jacobi structure (see App.~\ref{Exploding contact}). There is a Lie algebroid naturally associated to any Jacobi structure, and in particular to any Poisson structure.
\begin{definition}
The \emph{Jacobi Lie algebroid} \cite{c-z,k-s} of $(M,\Pi)$ is the bundle $T^*_\Pi M := T^*M\oplus \R$ with the anchor $\#(\alpha,f)=\#_\Pi\alpha$ and the bracket
\[
[(\alpha,f),(\beta,g)] = ([\alpha,\beta]_\Pi,\Pi(\alpha,\beta)) ,
\]
for $\alpha,\beta\in\Omega^1(M)$, $f,g\in\Ci(M)$.
\end{definition}
\begin{remark}
The Poisson bivector, $\Pi$, is a $2$-cocycle of the Lie algebroid $T^*M$, and $T^*_\Pi M$ is the central extension determined by it. My notation, $T^*_\Pi M$, is based on this fact.
\end{remark}
\begin{definition}
\label{Contact}
A \emph{contact form} on a $2n+1$-dimensional manifold is a $1$-form, $\theta$, such that $\theta\wedge(d\theta)^n$  is nowhere vanishing. 

A \emph{strict contact groupoid} $(\jac,\theta)$ is a Lie groupoid $\jac$ with a contact form $\theta\in\Omega^1(\jac)$ that is multiplicative:
\[
\partial^*\theta = \pr_1^*\theta-\m^*\theta+\pr_2^*\theta = 0 \ .
\]

A (strict) contact groupoid $(\jac,\theta)$ \emph{integrates} a Poisson manifold $(M,\Pi)$ as a Jacobi manifold if $\jac$ is an $\s$-connected groupoid over $M$ and $\t:\jac\to M$ is a Jacobi map in the sense that  for any $\gamma \in \jac$,  $\xi\in T^*_{\t\gamma} M$, and $X\in T_\gamma\jac$,
\[
X \inner d\theta = - \t^*\xi \implies T\t(X) = \#_\Pi \xi\ .
\]
\end{definition}
If $(\jac,\theta)$ integrates $(M,\Pi)$, then the Lie algebroid of $\jac$ is the Jacobi Lie algebroid, $T^*_\Pi M$. In fact, $(M,\Pi)$ is integrable to a contact groupoid if and only if $T^*_\Pi M$ is integrable  \cite{c-z}. 
In particular, the $\T$-bundle associated to a prequantization of a symplectic groupoid integrating $(M,\Pi)$ is a  contact groupoid integrating $(M,\Pi)$; see Section~\ref{Multiply}.
\begin{remark}
Contact groupoids that are not strict are discussed in Appendix~\ref{Exploding contact}. A contact groupoid integrating a Poisson manifold is always strict.
\end{remark}

\begin{definition}
\label{Heisenberg}
Let $\Pr_M:M\times \R\onto M$ denote the Cartesian projection and $\hbar$ the coordinate on $\R$.
Given a Poisson manifold, $(M,\Pi)$, the \emph{Heisenberg-Poisson manifold} \cite{wei2} is $\Hei(M,\Pi)=(M\times\R,\hbar\Pi)$. To be precise,  the Poisson structure is defined by 
\[
\{\Pr_M^*f,\Pr_M^*g\}_{\hbar\Pi} = \hbar \Pr_M^*\{f,g\}
\]
and
\[
\{\hbar,\Pr_M^*f\}_{\hbar\Pi}=0
\]
for $f,g\in\Ci(M)$.
For any Poisson map, $\phi$, define $\Hei\phi:=\phi\times\id_\R$. This $\Hei$ is a functor from the category of Poisson manifolds to the category of Poisson manifolds over $\R$. I also denote $\Hei M:=M\times\R$.
\end{definition}

Note that any cotangent vector on $\Hei M$ can be written as the sum of a multiple of $d\hbar$ and the pullback by $\Pr_M$ of a cotangent vector on $M$.

The Koszul bracket on $\Hei(M,\Pi)$ satisfies
\begin{align*}
[\Pr_M^*df,\Pr_M^*dg]_{\hbar\Pi} 
&= d(\hbar \Pr_M^*\{f,g\}) = \hbar \Pr_M^* d\{f,g\} + \Pr_M^*\{f,g\} d\hbar\\ 
&= \hbar \Pr_M^* [df,dg]_\Pi + \Pr_M^*[\Pi(df,dg)] d\hbar 
\end{align*}
and
\[
[d\hbar,\Pr_M^*df]_{\hbar\Pi} = 0 \ .
\]
By the properties of a Lie algebroid bracket this gives, for any $\alpha,\beta\in\Omega^1(M)$ 
\beq
\label{intertwine}
[\Pr_M^*\alpha,\Pr_M^*\beta]_{\hbar\Pi} = \hbar \Pr_M^* [\alpha,\beta]_\Pi + \Pr_M^*[\Pi(\alpha,\beta)] d\hbar 
\eeq
and
\[
[d\hbar,\alpha]_{\hbar\Pi} = 0 \ .
\]

This is very similar to the definition of the Jacobi bracket, except for the factor of $\hbar$. Let $\R_*$ be the set of nonzero real numbers. Over $M\times \R_* \subset \Hei M$, we can cancel that factor of $\hbar$ by rescaling the $1$-forms,
\[
[\hbar^{-1}\Pr_M^*\alpha,\hbar^{-1}\Pr_M^*\beta]_{\hbar\Pi} = \hbar^{-1} \Pr_M^* [\alpha,\beta]_\Pi + \Pr_M^*[\Pi(\alpha,\beta)] \hbar^{-2} d\hbar  .
\]
This shows that we can identify the Lie algebroid $T^*(M\times\R_*)$ with $T^*_\Pi M\times \R_*$. 
\begin{definition}
Let $\chi : T^*\Hei M\to T^*_\Pi M$ be the  vector bundle map over $\Pr_M:\Hei M \to M$ such that for any $\alpha\in\Omega^1(M)$, $\chi (\Pr_M^*\alpha) = (\hbar \alpha,0)$ and $\chi(d\hbar) = (0,\hbar^2)$.
\end{definition}
\begin{lem}
\label{Intertwine}
The map $\chi : T^*\Hei M\to T^*_\Pi M$ is a Lie algebroid homomorphism, and the restriction of $(\chi,\hbar)$ to $\hbar\neq0$ is an isomorphism from $T^*(M\times\R_*)$ to $T^*_\Pi M \times\R_*$.
\end{lem}
\begin{proof}
To show that $\chi$ is a homomorphism, it is sufficient to check that the base-preserving map $(\chi,\hbar) : T^*\Hei M\to T^*_\Pi M \times\R$ is a homomorphism, so  we just need to check that this intertwines the anchors and brackets. 

For any $\alpha\in\Omega^1(M)$,
\[
\#_{\hbar\Pi}(\Pr_M^*\alpha) = (\hbar\,\#_\Pi\alpha,0) = \#[\chi(\Pr_M^*\alpha)]
\]
and
\[
\#_{\hbar\Pi}d\hbar = 0 = \#[\chi(d\hbar)] \ .
\]

Equation \eqref{intertwine} shows that for $\alpha,\beta\in\Omega^1(M)$,
\begin{align*}
\chi([\Pr_M^*\alpha,\Pr_M^*\beta]_{\hbar\Pi}) &= \chi(\hbar \Pr_M^* [\alpha,\beta]_\Pi + \Pr_M^*[\Pi(\alpha,\beta)] d\hbar ) \\
&= (\hbar^2[\alpha,\beta]_\Pi,\hbar^2\Pi(\alpha,\beta))
= [(\hbar\alpha,0),(\hbar\beta,0)] \\
&= [\chi(\Pr_M^*\alpha),\chi(\Pr_M^*\beta)] \ .
\end{align*}
Likewise,
\[
\chi([d\hbar,\Pr_M^*\alpha]_{\hbar\Pi}) = 0 = [(0,\hbar^2),(\hbar\alpha,0)] = [\chi(d\hbar),\chi(\Pr_M^*\alpha)] \ .
\]

Over $\hbar\neq 0$, $(\chi,\hbar)$ is a bundle isomorphism, and therefore a Lie algebroid isomorphism.
\end{proof}

\subsection{Integration: motivation}
\label{Integration1}
The first objective here is to integrate $\Hei(M,\Pi)$, i.e., to construct a symplectic groupoid over $\Hei M$ (with Lie algebroid $T^*\Hei M$) such that the target map is a Poisson map. 
\begin{thm}
The Heisenberg-Poisson manifold $\Hei(M,\Pi)$ is integrable to a symplectic groupoid if and only $(M,\Pi)$ is integrable to a contact groupoid.
\end{thm}
\begin{proof}
Crainic and Fernandes \cite{c-f2} showed that a Poisson manifold is integrable if and only if its cotangent Lie algebroid is integrable. They also showed \cite{c-f1} that a Lie algebroid is integrable if and only if its  monodromy groups are locally uniformly discrete.  Crainic and Zhu \cite{c-z} showed that $(M,\Pi)$ is integrable to a contact groupoid if and only if $T^*_\Pi M$ is integrable.

For any Lie algebroid $A$ over $M$, denote the monodromy group of $A$ at $x\in M$ by $\mon_x(A)\subset A_x$. Monodromies are determined by maps of $S^2$ into orbits of $A$, so if $B\subset A$ is the full Lie subalgebroid over a submanifold containing $x$, then $\mon_x(A)=\mon_x(B)$.

In this case, $\hbar$ is constant along the orbits of $T^*\Hei M$ (which are  the symplectic leaves) so the restriction of $T^*\Hei M$ to a fixed  value of $\hbar$ is a full Lie subalgebroid, and Lemma~\ref{Intertwine} shows that this is isomorphic to $T^*_\Pi M$ for $\hbar\neq0$. To be precise, Lemma~\ref{intertwine} implies that for any $x\in M$ and $\hbar\neq0$, the monodromy group at $(x,\hbar$) is the inverse image by $\chi$: 
\begin{align*}
\mon_{(x,\hbar)}(T^*\Hei M) &= T^*_{(x,\hbar)}\Hei M \cap \chi^{-1}[\mon_x(T^*_\Pi M)] \\
&= \{\xi+a\,d\hbar \mid \xi\in T^*_xM,\ a\in\R,\  
(\hbar\xi,\hbar^2a)\in\mon_x(T^*_\Pi M)\}\ .
\end{align*}
On the other hand, the symplectic leaves of $\Hei(M,\Pi)$ over $\hbar=0$ are points, so 
\[
\mon_{(x,0)}(T^*\Hei M)=0\ .
\] 

Clearly, $\mon_{(x,\hbar)}(T^*\Hei M)$ is discrete for any $\hbar$ if and only if $\mon_x(T^*_\Pi M)$ is. Furthermore, the monodromies of $T^*_\Pi M$ can only be locally uniformly discrete if those of $T^*\Hei M$ are. Finally, if the monodromies of $T^*_\Pi M$ are locally uniformly discrete, then the monodromies of $T^*\Hei M$ expand away from $0$ as $\hbar\to0$, therefore they are also locally uniformly discrete.
\end{proof}

Since the objective is to integrate $\Hei(M,\Pi)$, we must assume that it is integrable, and so there exists a contact groupoid $(\jac,\theta)$ integrating $(M,\Pi)$ as a Jacobi manifold. 

Suppose that $(\Gamma,\omega_\Gamma)$ is a symplectic groupoid integrating $\Hei(M,\Pi)$, and suppose that the Lie algebroid homomorphism $\chi :T^*\Hei M \to T^*_\Pi M$ integrates to a groupoid homomorphism $\Chi:\Gamma\to\jac$ such that (like $\chi$) the restriction of $\Chi$ to any fixed $\hbar\neq0$ is an isomorphism.

This means that the $\hbar\neq0$ part of $\Gamma$ is isomorphic to $\jac\times\R_*$. This is a dense, open subgroupoid, so all we need now is to glue this together with a neighborhood of $\hbar=0$.

Let $\Exp_\jac:T^*_\Pi M \to \jac$ be some exponential map (Def.~\ref{Exponential}) that is well behaved on an open neighborhood $U\subset T^*_\Pi M$ of the $0$ section. Let $V := \chi^{-1}(U)$ and let $V_*\subset V$ be the $\hbar\neq0$ part.

Suppose that we can complete the commutative diagram
\[
\begin{tikzcd}
T^*\Hei M \arrow{r}{\chi} \arrow[dashed]{d}{\Exp_\Gamma} & T^*_\Pi M \arrow{d}{\Exp_\jac}\\
\Gamma \arrow{r}{\Chi} &\jac
\end{tikzcd}
\]
with an exponential map for $\Gamma$. The map $(\Exp_\jac\circ\chi,\hbar) : T^*\Hei M \to \jac\times\R$ is well behaved over $V_*$ (see Fig.~\ref{Domain}), so $\Exp_\Gamma$ is at least well behaved over $V_*$.

\begin{figure}
\begin{center}
\includegraphics{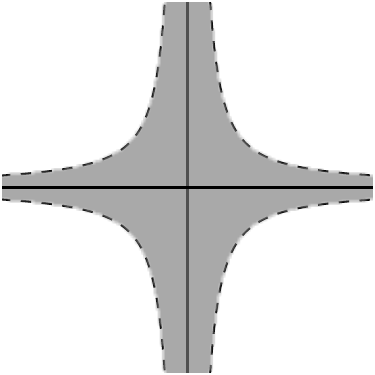}
\quad
\includegraphics{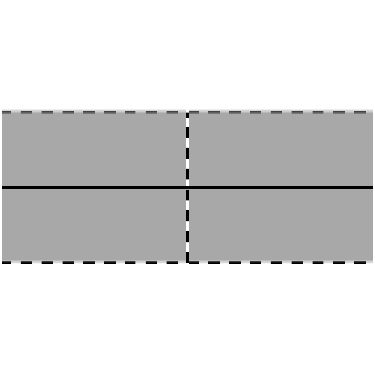}
\end{center}
\caption{Sketches of $V\subset T^*\Hei M$ with the $0$ section and the $\hbar=0$ part (left) and of its image $(\Exp_\jac\circ\chi,\hbar)(V) \subset \jac\times\R$ with the unit submanifold (right).\label{Domain}}
\end{figure}

$\Exp_\Gamma$ becomes better and better behaved as $\hbar\to0$. 
Over $\hbar=0$, $T^*\Hei M$ is a bundle of Heisenberg Lie algebras, so the restriction of $\Gamma$ should be a bundle of Heisenberg groups. Recall that the exponential map to a Heisenberg group is a diffeomorphism. This suggests that $\Exp_\Gamma$ is well behaved over all of $V$.

$V\subset T^*\Hei M$ is a neighborhood of the $\hbar=0$ part of $T^*\Hei M$ (see Fig.~\ref{Domain} again). If $\Exp_\Gamma V$ is a neighborhood of the $\hbar=0$ part of $\Gamma$, then we have a covering of $\Gamma$ by two open subsets. As a manifold, $\Gamma$ can be constructed as a union of $V$ and $\jac\times\R_*$, where $V_*\subset V$ is identified with a subset of $\jac\times\R_*$ by the map $(\Exp_\jac\circ\chi,\hbar)$. This is summarized in the push-forward diagram:
\begin{equation}
\label{tentative Gamma}
\begin{tikzcd}
V_* \arrow{r}{} \arrow{d}[swap]{(\Exp_\jac\circ \chi,\hbar)} &V \arrow[dashed]{d}{\Exp_\Gamma}\\
\jac\times\R_* \arrow[dashed]{r}{} & \Gamma \period
\end{tikzcd}
\end{equation}

As a set, $\Gamma$ is the disjoint union of $\jac\times\R_*$ with $T^*M\times\R$. To see the manifold structure, we look to coordinate charts.
Locally, a coordinate chart on $M$ and a local trivialization of $T^* M$ give a coordinate chart on $T^*\Hei M$ with four kinds of coordinates: $\hbar$, coordinates on $M$, coordinates in the fibers of $T^* M$, and the coefficient of $d\hbar$. This is glued together with $\jac\times\R_*$ using $\chi$, which gives an inhomogeneous rescaling: The coordinates on $M$ are not rescaled, the fiber coordinates are rescaled by $\hbar$, and the coefficient of $d\hbar$ is rescaled by $\hbar^2$.
In the terminology of Weinstein \cite{wei2}, this  is the double explosion of $\jac$ along $M$, in the direction of $T^*M$:
\[
\Gamma = \EE(\jac,M,T^*M)\ .
\]

To construct $\Gamma$ as a Lie groupoid, we need to smoothly extend the structure maps of $\jac\times\R_*\subset\Gamma$.

How about a symplectic form on $\Gamma$? The contact form $\theta$ on $\jac$ corresponds to the Poisson structure $\Pi$. If the Poisson structure is rescaled to $\hbar\Pi$, then the contact structure should be rescaled to $\hbar^{-1}\theta$. This suggests that $\hbar^{-1}\theta$ should be a symplectic potential on $\jac\times\R_*$. We shall see that this extends to a smooth form on $\Gamma$, and   $(\Gamma,-d[\hbar^{-1}\theta])$ is a symplectic groupoid integrating $\Hei(M,\Pi)$.

To prove this, it will be useful to have some general results about explosions.

\begin{remark}
Weinstein defined double explosions in order to construct a symplectic groupoid integrating the Heisenberg-Poisson manifold of a symplectic manifold, so it is not surprising that this construction is relevant to the more general case.
\end{remark}

\subsection{Explosions}
To begin, we need the category on which explosions are defined.
\begin{definition}
An \emph{explosive triple}  is a triple $(M,N,E)$, where $M$ is a (smooth) manifold, $N\subseteq M$ is a submanifold, and $E\subseteq \nu_N M = T M\rvert_N / TN$ a subbundle of the normal bundle. 
A \emph{compatible map} $\phi : (M,N,E)\to (M',N',E')$ between explosive triples is a smooth map $\phi : M \to M'$ such that $\phi(N)\subseteq N'$ and $T\phi(E)\subseteq E'$.
\end{definition}

Weinstein \cite{wei2} defines the \emph{double explosion} $\EE(M,N,E)$ of $M$ along $N$ in the direction of $E$ to be a union of $M\times \R_*$ with the normal bundle, glued together using an inhomogeneous rescaling. 
As we shall see, a compatible map, $\phi : (M,N,E)\to (M',N',E')$, determines a smooth map $\EE\phi : \EE(M,N,E)\to \EE(M',N',E)$ whose restriction to $M\times\R_*\subset \EE(M,N,E)$ equals $\phi\times\id_{\R_*}$. This makes $\EE$ a functor.

This is most easily constructed in the case that $M$ is an open subset of a vector space such that $N$ and $E$ are the intersections of vector subspaces with $M$. 
\begin{definition}
An \emph{explosive chart} is an explosive triple $(M,N,E)$ such that $M\subseteq \R^n$ is an open subset of a vector space with coordinates $(x,y,z)$ (each letter denotes a vector) such that
\[
N = \{(x,y,z)\in M \mid y=z=0\}
\]
and $E$ is spanned by the vectors in the $y$ directions.
\end{definition}
\begin{remark}
$N$ can be empty.
\end{remark}
\begin{definition}
\label{Double explosion}
The \emph{double explosion} of an explosive chart, $(M,N,E)$,  is
\beq
\label{double explosion}
\EE(M,N,E) := \{(x,y,z,\hbar)\in\R^{n+1} \mid (x,\hbar y,\hbar^2 z) \in M\} 
\eeq
and the projection $\Pr_M:\EE(M,N,E)\to M$ is 
\beq
\label{projection}
\Pr_M(x,y,z,\hbar) = (x,\hbar y,\hbar^2 z) \ .
\eeq
\end{definition}
\begin{thm}
\label{Explosion}
Let $\mathsf U$ be the forgetful functor from explosive triples to manifolds. There exists a functor, $\EE$, (which I call \emph{explosion}) from the category of explosive triples and compatible maps to the category of smooth manifolds over $\R$, and a natural transformation, $\Pr:\EE\naturalto \mathsf  U$ such that:
\begin{enumerate}
\item
For any explosive chart, $\EE$ is as defined by eq.~\eqref{double explosion}, $\Pr$  is as defined by \eqref{projection}, and the projection to $\R$ is the coordinate, $\hbar$;
\item
\label{Trivial}
for a trivial explosive triple, $\EE(M,M,0) = M\times \R$ and $\Pr_M$ is the Cartesian projection, and for a map $\phi :(M,M,0)\to(M',M',0)$, $\EE\phi = \phi\times\id_\R$;
\item
\label{Projection}
for any explosive triple, $(M,N,E)$, the restriction of 
\[
(\Pr_M,\hbar) : \EE(M,N,E) \to M\times \R
\]
to $\hbar\neq0$ is a diffeomorphism to $M\times\R_*$. 
\end{enumerate}
\end{thm}
\begin{proof}
I begin by proving this for the subcategory of explosive charts.

First, we need to define $\EE$ on morphisms. Let $\phi:(M,N,E)\to(M',N',E')$ be a compatible map of explosive charts. Naturality of $\Pr$ means that 
\beq
\label{natural}
\EE\phi\circ\Pr_M = \Pr_{M'}\circ \phi \ .
\eeq
That $\EE\phi$ is a map over $\R$ means $\hbar\circ\EE\phi = \hbar$. Putting these conditions together implies that $\EE\phi:\EE(M,N,E)\to \EE(M',N',E')$  must be a smooth map such that $\EE\phi(x,y,z,\hbar)=(x',y',z',\hbar)$, where 
\beq
\label{exploded map}
(x',\hbar y',\hbar^2 z') = \phi(x,\hbar y,\hbar^2 z) .
\eeq

Denote the components of $\phi$ as $\phi=(\phi^x,\phi^y,\phi^z)$ and denote derivatives by subscripts.  The condition that $\phi$ maps $N$ to $N'$ is simply $\phi^y(x,0,0)=\phi^z(x,0,0)=0$. This implies that the derivatives $\phi^y_x$ and $\phi^z_x$ similarly vanish. The condition that the induced map on the normal bundle maps $E$ to $E'$ means that also $\phi^z_y(x,0,0)=0$. In other words, the Jacobian matrix along $N$ reduces to the block triangular form:
\beq
\label{Tphi}
T\phi(x,0,0)=
\begin{pmatrix}
\phi^x_x & \phi^x_y & \phi^x_z \\
0 & \phi^y_y & \phi^y_z \\
0 & 0 & \phi^z_z 
\end{pmatrix} .
\eeq

L'H\^opital's rule shows that for $\EE\phi$ to be continuous, at $\hbar=0$ it must equal 
\beq
\label{Ephi}
\EE\phi(x,y,z,0) = (\phi^x,y\phi^y_y,\tfrac12 y^2\phi^z_{yy} + z\phi^z_z,0)\ ,
\eeq
where the functions on the right are all evaluated at $(x,0,0)$ and contraction is implied wherever possible between a coordinate and the corresponding derivative. Taylor's theorem shows that $\EE\phi$ is smooth, and so it is well defined.

This defines $\EE$ for compatible maps of explosive charts. $\Pr$ is natural by construction, and $\EE$ must be a functor because it is uniquely defined on maps by the naturality condition \eqref{natural}. Properties \ref{Trivial} and \ref{Projection} are immediate from the definitions of $\EE$ and $\Pr$ on explosive charts. This proves the theorem for the full subcategory of explosive charts.

In order to extend this to the whole category of explosive triples, observe that  constructing a manifold from a coordinate atlas really means expressing it as a colimit of coordinate charts and injective local diffeomorphisms. 

If a compatible map of explosive charts is just the inclusion of an open subset, then it is obvious from \eqref{exploded map} that the explosion is also an inclusion. If a compatible map of explosive charts has a compatible inverse, then by functoriality, its explosion is a diffeomorphism. Any injective local diffeomorphism of explosive charts is a composition of such maps, therefore its explosion is an injective local diffeomorphism.

If $(M,N,E)$ is an explosive triple, then there exists an atlas for $M$ of explosive charts. Applying $\EE$ to the atlas gives an atlas, which has a colimit. Define $\EE(M,N,E)$ as some colimit of this atlas. The naturality properties of colimits imply that $\EE$ extends to a functor on the category of explosive triples.
\end{proof}

\begin{remark}
I am using the imprecise notation $\Pr_M$, because the correct notation $\Pr_{(M,N,E)}$ is too cumbersome.
\end{remark}

\begin{remark}
The map $(\Pr_M,\hbar):\EE(M,N,E)\to M\times\R$ identifies a dense subset of $\EE(M,N,E)$ with $M\times\R_*$. The natural projection, $\Pr_M$ is just the smooth extension of the Cartesian projection  $M\times\R_*\to M$. The explosion $\EE\phi$ is just the completion of the Cartesian product of $\phi$ with the identity map on $\R_*$. The projection $\hbar:\EE(M,N,E)\to\R$ is just the explosion of the unique map from $M$ to a point.
\end{remark}

\begin{remark}
$N\subseteq \nu_N M$ can be thought of as a (very short) filtration of the normal bundle $\nu_N M$. Indeed, it also gives a filtration of $TM\rvert_N$.
\end{remark}
\begin{cor}
\label{h0}
The restriction of $\EE(M,N,E)$ to $\hbar=0$ is the graded vector bundle (over $N$) associated to this filtration of the normal bundle. The restriction of $\EE\phi$ to $\hbar=0$ is given in explosive charts by eq.~\eqref{Ephi}.
\end{cor}

\begin{definition}
A compatible map $\phi : (M,N,E) \to (M',N',E')$ is an \emph{explosive submersion} if $\phi:M\to M'$ is a submersion, the restriction $\phi\rvert_N : N\to N'$ is a submersion, and the induced map from $E$ to $E'$ is fiberwise surjective.

A compatible map $\phi : (M,N,E) \to (M',N',E')$ is an \emph{explosive immersion} if $\phi$ is an immersion, the induced map on normal bundles is fiberwise injective, and the induced map from $\nu_N M/E$ to $\nu_{N'}M'/E'$ is fiberwise injective.
\end{definition}

\begin{lem}
If $\phi$ is an explosive submersion, then $\EE\phi$ is a submersion.
If $\phi$ is an explosive immersion, then $\EE\phi$ is an immersion.
If $\phi$ is an explosive immersion and is injective, then $\EE\phi$ is also injective.
\end{lem}
\begin{proof}
These are obvious over the $\hbar\neq0$ regions, so we just need to check what happens at $\hbar=0$. It is sufficient to prove these claims for explosive charts. All derivatives of $\phi$ in this proof are evaluated at $(x,0,0)$.

Further application of L'H\^opital's rule shows that the Jacobian matrix of $\EE\phi$ at $\hbar=0$ is
\beq
\label{TEphi}
T\EE\phi(x,y,z,0) = 
\begin{pmatrix}
\phi^x_x & 0 & 0 & y\phi^x_y \\
y \phi^y_{xy} & \phi^y_y & 0 & \frac12y^2\phi^y_{yy}+z\phi^y_z \\
\frac12y^2\phi^z_{xyy}+z\phi^z_{xz} & y\phi^z_{yy} & \phi^z_z & \frac16 y^3\phi^z_{yyy} + yz \phi^z_{yz} \\
0 & 0 & 0 & 1
\end{pmatrix}\ .
\eeq 
If $w = T\EE\phi(x,y,z,0)v$, then (with a slight reordering of the components)
\beq
\label{wvrelation}
\begin{pmatrix}
w^\hbar \\ w^x \\ w^y \\ w^z
\end{pmatrix}
=
\begin{pmatrix}
1 & 0 & 0 & 0 \\
y\phi^x_y  &\phi^x_x & 0 & 0\\
\frac12y^2\phi^y_{yy}+z\phi^y_z  & y \phi^y_{xy} & \phi^y_y & 0 \\
 \frac16 y^3\phi^z_{yyy} + yz \phi^z_{yz} &\frac12y^2\phi^z_{xyy}+z\phi^z_{xz} & y\phi^z_{yy} & \phi^z_z 
\end{pmatrix}
\begin{pmatrix}
v^\hbar \\ v^x \\ v^y \\ v^z
\end{pmatrix}
\eeq

First, suppose that $\phi$ is an explosive submersion. This means, in particular, that the Jacobian $T\phi(x,0,0)$ (given by the matrix \eqref{Tphi}) is surjective. In particular, any value of the $z$-components occurs in the image, so $\phi^z_z$ must be surjective.
The restriction $\phi\rvert_N$ is given by $\phi^x(x,0,0)$, so the condition that this is a submersion means precisely that the matrix $\phi^x_x$ is surjective.
The induced map from $E$ to $E'$ is the matrix $\phi^y_y$, therefore this is surjective.
In other words, these conditions on $\phi$ imply that the block diagonal components of $T\EE\phi(x,y,z,0)$ in \eqref{TEphi} are each surjective. 

To show that the whole matrix is surjective, we need to show that for any $w$, eq.~\eqref{wvrelation} has a solution for $v$. Because the matrix is block triangular, we can iteratively solve for $v^\hbar$, $v^x$, $v^y$, and $v^z$, using $v^\hbar=w^\hbar$ and the surjectivity of $\phi^x_x$, $\phi^y_y$, and $\phi^z_z$.

Now, suppose instead that $\phi$ is an explosive immersion. In particular, $\phi\rvert_N$ is an immersion, and so $\phi^x_x$ is injective. The induced map on normal bundles is given by the matrix
\[
\begin{pmatrix}
\phi^y_y & \phi^y_z \\
0 & \phi^z_z
\end{pmatrix} ;
\]
that this is injective implies that $\phi^y_y$ is injective. The induced map on the quotients by $E$ and $E'$ is given by $\phi^z_z$, therefore this is an injective matrix, so all of the block diagonal components of $T\EE\phi(x,y,z,0)$ \eqref{TEphi} are injective.

If $T\EE\phi(x,y,z,0)v=0$, then $\phi^z_zv^z=0 \implies v^z=0 \implies \phi^y_y v^y =0 \implies v^y=0 \implies \phi^x_x v^x=0 \implies v^\hbar=0$, so $v=0$.
Therefore $T\EE\phi(x,y,z,0)$ is injective and $\EE\phi$ is an immersion.

Now suppose that $\phi$ is an injective explosive immersion. 
Let $(x,y,z,0)$ and $(x',y',z',0)$ be two points such that,
\beq
\label{injective}
\EE\phi(x,y,z,0)=\EE\phi(x',y',z',0) \ .
\eeq
$\EE\phi$ at these points is given explicitly by \eqref{Ephi}.
Firstly, the $x$-component of eq.~\eqref{injective} states that $\phi^x(x,0,0)=\phi^x(x',0,0)$; injectivity of $\phi$ implies that $x=x'$. Secondly, the $y$-component states that $y\phi^y_y=y'\phi^y_y$, but injectivity of $\phi^y_y$ implies that $y=y'$. Finally, the $z$-component states that 
\[
\tfrac12 y^2\phi^z_{yy} + z\phi^z_z = \tfrac12 y^2\phi^z_{yy} + z'\phi^z_z 
\]
thus $(z-z')\phi^z_z=0$, and injectivity of $\phi^z_z$ implies that $z=z'$. Therefore, $(x,y,z,0)=(x',y',z',0)$ and so $\EE\phi$ is injective.
\end{proof}

\subsubsection{Exploding forms}
In order to construct the symplectic potential in the next section, it will be useful to construct differential forms on explosions.
\begin{definition}
Given an explosive triple, $(M,N,E)$, let $\Omegan(M,N,E)\subset\Omega^1(M)$ be the space of $1$-forms that are normal to $N$ and $E$ along $N$.
\end{definition}
It is easy to check that the pullback of such a form by a compatible map also satisfies this condition, therefore $\Omegan$ is a contravariant functor from the category of explosive triples to the category of vector spaces.
\begin{lem}
\label{Exploding forms}
There exists a unique natural tranformation $\EA  : \Omegan \naturalto \Omega^1\circ\EE$, such that for any $\theta\in\Omegan(M,N,E)$, $\hbar\EA_M \theta = \Pr_M^*\theta$.
\end{lem}
\begin{proof}
First, consider the subcategory of explosive charts and work with explicit coordinates. Denote components by subscripts. 
For $\theta = \theta_xdx+\theta_ydy+\theta_zdz$, $\Pr_M^*\theta$ is explicitly
\begin{multline*}
(\Pr_M^*\theta)(x,y,z,\hbar) = \\ \theta_x(x,\hbar y,\hbar^2 z) dx + \theta_y(x,\hbar y,\hbar^2 z) d(\hbar y) + \theta_z(x,\hbar y,\hbar^2 z) d(\hbar^2z) \ .
\end{multline*}
The first two terms vanish at $\hbar=0$ because $\theta\in\Omegan(M,N,E)$ means precisely that $\theta_x(x,0,0)=\theta_y(x,0,0)=0$. 
The last term vanishes at $\hbar=0$ because $d(\hbar^2 z)=\hbar^2dz+2\hbar\, z\,d\hbar$. This shows that $\Pr_M^*\theta$ vanishes at $\hbar=0$ and so must be a multiple of $\hbar$, therefore there exists a unique $\EA_M\theta$ such that $\hbar\EA_M\theta = \Pr_M^*\theta$.

Naturality of $\Pr$ means that for any compatible map $\phi:(M,N,E)\to (M',N',E')$, there is a commutative diagram,
\[
\begin{tikzcd}
\EE(M,N,E) \arrow{r}{\EE\phi} \arrow{d}{\Pr_{M}} &\EE(M',N',E')\arrow{d}{\Pr_{M'}} \\
M \arrow{r}{\phi} &M' \period
\end{tikzcd}
\] 
Pulling back some $\theta\in\Omegan(M',N',E')$ along this diagram gives that
\[
\hbar\, (\EE\phi)^*(\EA_{M'} \theta) = (\EE\phi)^*\Pr_{M'}^*\theta \stackrel!= \Pr_{M}^*\phi^*\theta = \hbar\EA_{M}(\phi^*\theta) ,
\] 
and therefore $(\EE\phi)^*\EA_{M'} \theta = \EA_{M}(\phi^*\theta)$, i.e., $\EA$ is natural. This naturality implies that the construction is compatible with gluing in an atlas, and so it extends to the full category of explosive triples.
\end{proof}

\subsubsection{Simple explosions}
Weinstein also defined the simple explosion of a manifold along a submanifold. These had also been previously constructed in \cite{h-s}.
\begin{definition}
An \emph{explosive pair} is a pair $(M,N)$, where $M$ is a manifold and $N\subseteq M$ is a submanifold; this is identified  with the explosive triple $(M,N,E)$, where $E=\nu_N M$ is the entire normal bundle. The \emph{simple explosion} of an explosive pair $(M,N)$ is $\EE(M,N) := \EE(M,N,\nu_N M)$.
\end{definition}
I will regard explosive pairs as a full subcategory of explosive triples. 

The construction of simple explosions is simpler than the general construction of double explosions because there are no $z$  coordinates or rescaling by $\hbar^2$. In particular, the restriction of $\EE(M,N)$ to $\hbar=0$ is the normal bundle $\nu_N M$, and the restriction of $\EE\phi$ to $\hbar=0$ is the linearization of $\phi$ about $N$.

Note that if $(M,N)$ is an explosive pair, then $(TM,TN)$ is also an explosive pair, and the simple explosion is the subbundle $\EE(TM,TN)\subset T\EE(M,N)$ of vectors tangent to the constant $\hbar$ hypersurfaces. The bundle projection is the explosion of the bundle projection for $TM$.

In order to construct polarizations in Section \ref{Polarization}, we will need to construct tangent distributions on simple explosions. 
\begin{lem}
\label{Lift distribution}
If $(M,N)$ is an explosive pair and $\F\subset TM$ is a tangent distribution (i.e., a subbundle) such that $\rk (\F\cap TN)$ is locally constant, then $\EE(\F,\F\cap TN) \subset T\EE(M,N)$ is a tangent distribution. The same is true for a complex tangent distribution, $\F\subset T_\co M$.
\end{lem}
\begin{proof}
The intersection $\F\cap TN$ is the kernel of the map $\F\rvert_N\oplus TN \to T_NM$, defined by subtraction. This map has constant rank, therefore $\F\cap TN$ is locally trivial, hence it is a submanifold of $TN$, and $(\F,\F\cap TN)$ is indeed an explosive pair. The inclusion $\F\subset TM$ is an injective explosive immersion; its explosion gives the inclusion $\EE(\F,\F\cap TN)\subset \EE(TM,TN) \subset T\EE(M,N)$.

Because $\F$ and $\F\cap TN$ are locally trivial, any point of $M$ has a neighborhood $U\subset M$, such that there exists a trivialization of $\F$ over $U$ that restricts to a trivialization of $\F\cap TN$ over $U\cap N$. The explosion of such a local trivialization gives a local trivialization of $\EE(\F,\F\cap TN)$. 

The proof for $\F\subset T_\co M$ is identical.
\end{proof}

\subsection{Integration: construction}
\label{Integration2}
The definitions of the previous section allow a more precise statement of the guesses of Section~\ref{Integration1}. Again, let $(\jac,\theta)$ be a contact groupoid that integrates a Poisson manifold $(M,\Pi)$.

The unit map $\unit : M\into \jac$ identifies $M$ with a submanifold of $\jac$. The normal bundle  $\nu_M\jac$ is the Lie algebroid $T^*_\Pi M$. In this sense, $T^* M$ is a subbundle of the normal bundle: $T^*M\subset  T^*_\Pi M = T^*M\oplus\R$. In this way, we have an explosive triple, $(\jac,M,T^*M)$.
\begin{definition}
Let $\Gamma := \EE(\jac,M,T^*M)$ with structure maps defined by applying $\EE$ to the structure maps of $\jac$. Let $\omega_\Gamma := -d(\EA_\jac\theta) \in \Omega^2(\Gamma)$.

Let $n:= \dim M$, so that $\dim \jac = 2n+1$ and $\dim\Gamma=2n+2$. 
\end{definition}

\begin{lem}
\label{Groupoid}
$\Gamma$ is a Lie groupoid over $\Hei M$.
\end{lem}
\begin{proof}
Note that $(M,M,0)$ is an explosive triple. By Theorem~\ref{Explosion}(\ref{Trivial}), the explosion is just 
\[
\EE(M,M,0) = M\times\R = \Hei M
\]
as a manifold.  The map $(\unit,\unit) : M\to \jac_2$ identifies $M$ with a submanifold of $\jac_2$. The normal bundle is $T^*_\Pi M\oplus T^*_\Pi M$. In terms of this, $(\jac_2,M,T^*M\oplus T^*M)$ is an explosive triple;
denote its explosion tentatively as $\Gamma_2$. 

The groupoid structure maps $\unit : M\to\jac$, $\inv :\jac\to\jac$, $\s,\t:\jac\to M$, and $\m,\pr_1,\pr_2 : \jac_2 \to \jac$ are all compatible with these explosive triples, so we can explode all of them. 
In particular, $\s$ and $\t$ are explosive submersions, so $\EE\s$ and $\EE\t$ are submersions as well.

To check that the fibers of $\EE\t$ are Hausdorff, we need to show that any two points of a fiber are either separated or equal, so suppose that $\gamma,\eta\in\Gamma$ such that $\EE\t(\gamma)=\EE\t(\eta)$. Firstly, $\Pr_\jac\gamma$ and $\Pr_\jac\eta$ are in the same $\t$-fiber, so if $\Pr_\jac\gamma\neq\Pr_\jac\eta$, then they are separated, and the inverse images by $\Pr_\jac$ of their disjoint neighborhoods are disjoint neighborhoods of $\gamma$ and $\eta$, so those are separated. Secondly, if $\Pr_\jac\gamma=\Pr_\jac\eta$ and $\hbar(\gamma)\neq\hbar(\eta)$, then they are separated, but if $\hbar(\gamma)=\hbar(\eta)\neq0$, then they are equal. Finally, if $\Pr_\jac\gamma=\Pr_\jac\eta$ and $\hbar(\gamma)=\hbar(\eta)=0$, then $\gamma$ and $\eta$ lie in the same coordinate patch in the construction of $\Gamma$, so they are either separated or equal.

We do need to check that $\Gamma_2$ really is the space of composable pairs, $\Gamma \mathbin{{}_{\EE\s}\times_{\EE\t}} \Gamma$, so let 
\[
F := (\EE\pr_1,\EE\pr_2) : \Gamma_2 \to \Gamma \mathbin{{}_{\EE\s}\times_{\EE\t}} \Gamma \subset \Gamma\times\Gamma \ .
\]
By construction, the restriction of $F$ to $\hbar\neq0$ is a diffeomorphism, so consider what happens at $\hbar=0$. By Corollary~\ref{h0}, the restriction of $\Gamma$ to $\hbar=0$ is the bundle $T^*M\oplus\R$ and $\EE\s$ and $\EE\t$ both restrict to the bundle projection to $M$. This shows that the restriction of $\Gamma \mathbin{{}_{\EE\s}\times_{\EE\t}} \Gamma$ to $\hbar=0$ is the bundle $T^*M\oplus\R\oplus T^*M\oplus\R$. This is also the restriction of $\Gamma_2$, so the restriction of $F$ to $\hbar=0$ is a diffeomorphism. This implies that  $F$ is bijective and the differential of $F$ is bijective, therefore $F$ is a diffeomorphism. 

There is a similar explosive triple for $\jac_3$, and by a similar argument that explosion is exactly the space of composable triples.

All of the remaining axioms for a Lie groupoid are expressed as commutative diagrams in Definition~\ref{Groupoid def}. Each of these diagrams is commutative for $\jac$, and applying the functor $\EE$ gives the corresponding commutative diagram for $\Gamma$.
\end{proof}
\begin{remark}
To see that $\Gamma$ really is the push-forward in \eqref{tentative Gamma},
note that a coordinate chart on $M$ and a trivialization of $T^*_\Pi M$ over the coordinate patch give an explosive chart on $\jac$ with the help of $\Exp_\jac:T^*_\Pi M\to\jac$. Covering $\jac$ with such charts gives an atlas that expresses this union as an explosion.
\end{remark}
\begin{remark}
Although the subgroupoid of $\Gamma$ at $\hbar=0$ is diffeomorphic to $T^*M\times\R$, the groupoid structure is not the obvious one (a bundle of abelian groups). Instead,  $T^*M\times\R$ is a central extension of  the bundle of abelian groups, $T^*M$, determined by $\Pi$. The fibers are Heisenberg groups. This can be seen by considering the Lie algebroid $T^*\Hei M$, which restricts to a bundle of Lie algebras along $\hbar=0$. This is discussed further in Appendix~\ref{Parabolic}.
\end{remark}

\begin{lem}
\label{Symplectic potential}
The contact form, $\theta$, is in $\Omegan(\jac,M,T^*M)$, so 
$\EA_\jac \theta\in\Omega^1(\Gamma)$ is defined. Moreover, $\EA_\jac \theta$ is multiplicative. 
\end{lem}
\begin{proof}
By definition, the contact form $\theta\in\Omega^1(\jac)$ is multiplicative; that is
\beq
\label{multiplicativity}
0 = \partial^*\theta = \pr_1^*\theta - \m^*\theta + \pr_2^*\theta .
\eeq

The composition of $\pr_1$, $\pr_2$ or $\m$ with $(\unit,\unit)$ is just $\unit$, so pulling back eq.~\eqref{multiplicativity} by $(\unit,\unit)$ gives
\[
0 = \unit^*\theta - \unit^*\theta + \unit^*\theta = \unit^*\theta\ .
\]
This means that $\theta$ is normal to the unit submanifold. Any multiplicative $1$-form is equivalent to an algebroid $1$-cochain, which is just the pairing with the normal bundle. In this case, the algebroid cochain maps any element of $T^*_\Pi M = T^*M\oplus\R$ to its $\R$-component. This means that $T^*M\subset T^*_\Pi M$ is precisely the subbundle of the normal bundle that is normal to $\theta$, therefore $\theta\in\Omegan(\jac,M,T^*M)$, and $\EA_\jac\theta$ is defined.

By eq.~\eqref{multiplicativity} and Lemma \ref{Exploding forms} (naturality of $\EA$)
\begin{multline*}
(\EE\pr_1)^*\EA_\jac \theta - (\EE\m)^*\EA_\jac \theta + (\EE\pr_2)^*\EA_\jac \theta 
 = \EA_{\jac_2} (\pr_1^*\theta - \m^*\theta + \pr_2^*\theta ) =0
\end{multline*}
so $\EA_\jac \theta$ is multiplicative. 
\end{proof}

\begin{thm}
\label{Integration thm}
$(\Gamma,\omega_\Gamma)$  is a symplectic groupoid integrating $\Hei(M,\Pi)$.
\end{thm}
\begin{proof}
Since $\theta$ is a contact form, $\theta\wedge (d\theta)^n$ is a nowhere-vanishing volume form on $\jac$.
Over $\hbar\neq0$, the maximum exterior power of 
\[
\omega_\Gamma=-d(\EA_\jac\theta) = \hbar^{-2}d\hbar\wedge\Pr_\jac^*\theta-\hbar^{-1}\Pr_\jac^*d\theta
\] 
is
\beq
\label{volume}
\omega_\Gamma^{n+1} = (-1)^{n+1}(n+1)\hbar^{-n-2}\Pr_\jac^*(\theta\wedge [d\theta]^n)\wedge d\hbar .
\eeq
In an explosive chart around $M\subset\jac$, there are $n$ $x$-coordinates, $n$ $y$-coordinates, and one $z$-coordinate. Equation~\eqref{projection} shows that the Jacobian of the map $(\Pr_\jac,\hbar)$ is 
$\hbar^{n+2}$, which exactly cancels the power of $\hbar$ in eq.~\eqref{volume}, therefore $\omega_\Gamma^{n+1}$ is nonvanishing, and $\omega_\Gamma$ is nondegenerate. As $\omega_\Gamma$ is manifestly closed, it is a symplectic form. Multiplicativity of $\EA_\jac\theta$ (Lem.~\ref{Symplectic potential}) implies that $\omega_\Gamma$ is multiplicative, and therefore $(\Gamma,\omega_\Gamma)$ is a symplectic groupoid.

Let $\Hei_+ M:=M\times\R_+\subset \Hei M$ be the open submanifold with $\hbar>0$. A Poisson manifold is a special type of  Jacobi manifold, and in \cite{lic} (see also \cite{c-z}) a construction is given for the ``Poissonization'' of a Jacobi manifold. The Poissonization of $(M,\Pi)$ is simply $\Hei_+(M,\Pi):=(\Hei_+ M,\hbar\Pi)$; the explicit constructions are related by the coordinate transformation $\hbar=e^{-s}$.

A contact manifold is also a special type of Jacobi manifold. The Poissonization $\Hei_+(\jac,\theta)$ is a symplectic groupoid integrating $\Hei_+(M,\Pi)$ \cite[Prop.~2.7]{c-z}. As a symplectic groupoid, this is precisely the part of $(\Gamma,\omega_\Gamma)$ where $\hbar>0$.

This shows that the restriction of the target map $\EE\t:\Gamma\to \Hei M$ to $\hbar>0$ is a Poisson map. By changing signs, this also shows that the same is true for $\hbar<0$. Finally, $\hbar\neq0$ over a dense subset, therefore $\EE\t:\Gamma\to \Hei M$ is a Poisson map.
\end{proof}

\begin{remark}
Given a contact groupoid integrating a Jacobi manifold, this construction generalizes to give a symplectic groupoid integrating the Heisenberg-Poisson manifold of the Jacobi manifold. This is described in Appendix~\ref{Exploding contact}.
\end{remark}

\section{Quantization of Heisenberg-Poisson manifolds}
\label{Quantization HP}
Suppose that $(M,\Pi)$ is a Poisson manifold for which both Lie algebroids, $T^*M$ and $T^*_\Pi M$, are integrable. This implies that there exist both a symplectic groupoid $(\Sigma,\omega)$ and a  contact groupoid $(\jac,\theta)$ integrating $(M,\Pi)$, along with a surjective homomorphism (a \emph{fibration of groupoids }\cite{mac3}) $\pro : \jac\onto\Sigma$  such that
\beq
\label{symplectic quotient}
\pro^*\omega = - d\theta
\eeq
(compare \cite[Thm.~2]{c-z}).

If we start with the contact groupoid, $(\jac,\theta)$, then the kernel of $d\theta$ is a rank $1$ foliation of $\jac$. If  $\Sigma$ is defined to be the leaf space and $\pro:\jac\to\Sigma$ the quotient map, then (provided that $\Sigma$ is smooth) eq.~\eqref{symplectic quotient} determines a symplectic groupoid structure integrating $(M,\Pi)$.

Conversely, if we start with the symplectic groupoid $(\Sigma,\omega)$ and a prequantization $(L,\nabla,\sigma)$, then we can define $\jac$ to be the circle bundle associated to $L$. The cocycle $\sigma$ determines the groupoid product on $\jac$, and the connection $1$-form associated to $\nabla$ is the contact form.

Throughout this section,  $(\Gamma,\omega_\Gamma)=(\EE[\jac,M,T^*M],-d[\EA_\jac\theta])$ is the symplectic groupoid constructed from $(\jac,\theta)$. Structure maps without subscripts ($\unit$, $\s$, $\t$, $\inv$, and $\m$) are those for $\jac$. The dimensions  are $\dim M = n$, $\dim \Sigma = 2n$, $\dim \jac = 2n+1$, and $\dim\Gamma = 2n+2$.

\subsection{Prequantization}
\label{Prequantization}
Because the symplectic form $\omega_\Gamma$ is exact, $(\Gamma,\omega_\Gamma)$ can be prequantized with the trivial line bundle $\co\times\Gamma$  over $\Gamma$. Because $\EA_\jac\theta$ is a symplectic potential, the connection $d+i\EA_\jac \theta$ has curvature $\omega_\Gamma$. Because $\EA_\jac\theta$ is real, this is compatible with the trivial Hermitian inner product on $\co\times\Gamma$. Because $\EA_\jac\theta$ is multiplicative, the trivial cocycle $1$ is compatible with this connection.

In this way, $(\co\times\Gamma,d+i\EA_\jac \theta,1)$ is a prequantization of the symplectic groupoid $(\Gamma,\omega_\Gamma)$. However, there are other equivalent choices of prequantization that will be useful in more specific examples.

\subsection{Polarization}
Let $\P$ be some polarization of the symplectic groupoid $\Sigma$. Let $\P_0$ and $\D_0$ be as in Definition~\ref{Distributions}.

\begin{definition}
$\tilde\P := \EE(\P,\P\cap T_\co M)\subset T_\co \EE(\Sigma,M)$ and $\Q := (T\EE\pro)^{-1}(\tilde\P)$.
\end{definition}

\begin{lem}
The simple explosion $\EE(\Sigma,M)$ is a Lie groupoid over $\Hei M$.
If $\rk \D_0$ is locally constant, then $\tilde\P$ is a polarization of  $\EE(\Sigma,M)$.
\end{lem}
\begin{proof}
The proof that $\EE(\Sigma,M)$ is a Lie groupoid is just a simpler version of the proof of Lemma~\ref{Groupoid}. I will not repeat the details.

Let $x\in M$ and $v \in \D_{0\,x}^\co$. By the definition of $\D_0$ there exist $X\in\P_{\unit(x)}$ and $Y\in\bar\P_{\unit(x)}$ such that $v = T\t(X)=T\t(X)$. 
The groupoid multiplication in $T_{\unit(x)}\Sigma$ (denoted by a dot here) is linear (rather than bilinear), therefore
\[
X\cdot Y^{-1} = (X-T\unit[v])\cdot 0 + (T\unit[v])\cdot(T\unit[v]) + 0\cdot (Y^{-1}-T\unit[v]) = X+Y^{-1}-T\unit(v)\ .
\]
The Hermiticity of $\P$ implies that $Y^{-1}\in\P_{\unit(x)}$.
The multiplicativity of $\P$ implies that $X\cdot Y^{-1}\in\P_{\unit(x)}$, and since this is a vector space, $T\unit(v) \in \P$. This shows that $T\unit(\D_0^\co) \subseteq\P\cap T_\co M$. The opposite inclusion is immediate from the definition of $\P_0$, therefore $\P\cap T_\co M = T\unit(\D_0^\co)$ has constant rank and is locally trivial, and by Lemma \ref{Lift distribution}, $\tilde\P$ is locally trivial.

The restriction of $\tilde\P$ to $\hbar\neq0$ is $\P\times\R_*$, which is a groupoid polarization of $\Sigma\times\R_*$. 
In particular, $\P\times\R_*$ is involutive, so if $X,Y\in\Gamma(\EE(\Sigma,M),\tilde\P)$, then the Lie bracket $[X,Y]$ is valued in $\tilde\P$ for $\hbar\neq0$. By continuity, $[X,Y]\in\Gamma(\EE(\Sigma,M),\tilde\P)$, therefore $\tilde\P$ is involutive.

By Corollary \ref{h0}, the restriction of $\EE(\Sigma,M)$ to $\hbar=0$ is the vector bundle $T^*M$, with the groupoid structure of fiberwise addition. A tangent distribution on a vector bundle is multiplicative if it is vertically constant. The restriction of $\tilde\P$ is vertically constant, therefore $\tilde\P$ is multiplicative. It is also invariant under complex conjugation and negation on $T^*M$. Therefore $\tilde\P$ is a groupoid polarization of $\EE(\Sigma,M)$.
\end{proof}

\begin{thm}
\label{Polarization}
If $\rk \D_0$ is locally constant, then $\Q$ is a polarization of the symplectic groupoid $(\Gamma,\omega_\Gamma)$.
\end{thm} 
\begin{proof}
The quotient map $\pro :(\jac,M,T^*M)\to(\Sigma,M)$ is an explosive submersion, and so $\EE\pro : \Gamma\to \EE(\Sigma,M)$ is a submersion, therefore $\Q$ is locally trivial. By \cite[Lem.~7.10]{haw10} $\Q$ is a groupoid polarization of $\Gamma$.
It remains to check that $\Q$ is a Lagrangian distribution.

Over $\jac\times\R_*\subset\Gamma$, $\Q$ is equal to $(T\pro)^{-1}(\P)\times0$. 
Let $X,Y\in \Q$ be two vectors over the same point of $\jac\times\R_*$. The directional derivative of $\hbar$ by either of these vectors is $0$, therefore the contraction with the symplectic form is
\[
\omega_\Gamma(X,Y) = (-d[\hbar^{-1}\theta])(X,Y) = -\hbar^{-1} d\theta(X,Y) .
\]
Now, we can view $X$ and $Y$ as vectors on $\jac$. The relation \eqref{symplectic quotient} between the symplectic and contact forms gives
\[
-d\theta(X,Y) = (\pro^*\omega)(X,Y) = \omega(T\pro[X],T\pro[Y]) =0 ,
\]
because $T\pro(X), T\pro(Y)\in\P$, which is Lagrangian. By continuity, this shows that $\Q$ is isotropic, but because $\rk\Q=n+1$, it is Lagrangian.
\end{proof}
\begin{remark}
If $\rk \D_0$ is not locally constant, then it seems reasonable to regard $\Q$ as a ``singular polarization''\/. This can probably still be used for quantization. In a recent paper, Bonechi, Ciccoli, Staffolani, and Tarlini \cite{b-c-s-t} used a singular polarization to construct the  \Podles\ standard sphere by the symplectic groupoid approach to quantization.
\end{remark}

\subsection{Half-forms}
\label{Half-forms}
Unless $\P$ is a real polarization, $\Q^\t$ will not be locally trivial at $\hbar=0$, so the definition of $\Omega_\Q$ by eq.~\eqref{form bundle} doesn't make sense there.

Over $\hbar\neq0$, $\tilde\P = \P\times0$, so
\[
T^\t_\co\Gamma/\Q^\t = \EE\pro^*[T^\t_\co\EE(\Sigma,M)/\tilde\P^\t]
= \EE\pro^*\Pr_\Sigma^*(T^\t_\co\Sigma/\P^\t)
= \Pr_\jac^*\pro^*(T^\t_\co\Sigma/\P^\t)\ .
\]
This means that over $\hbar\neq0$, $\Omega_\Q = \Pr_\jac^*\pro^*\Omega_\P$, so it is tempting to define $\Omega_\Q$ to be $\Pr_\jac^*\pro^* \Omega_\P$ over all of $\Gamma$. 

\begin{definition}
Let $\Gamma\rvert_{\hbar=0}$ be the $\hbar=0$ subgroupoid and $\iota : \Gamma\rvert_{\hbar=0}\into \Gamma$ the inclusion. 
\end{definition}
The polarization $\Q$ restricts to a polarization $\iota^*\Q$ of $\Gamma\rvert_{\hbar=0}$, which in turn defines an  $\Omega_{\iota^*\Q}$. Unfortunately, this is not canonically the same as $\iota^*\Pr_\jac^*\pro^* \Omega_\P$, although they are actually isomorphic. We need a way of gluing together $\Omega_{\iota^*\Q}$ at $\hbar=0$ and $\Pr_\jac^*\pro^* \Omega_\P$ over $\hbar\neq0$.

This is provided by $\varepsilon_\Q$ (Def.~\ref{Epsilon factor}) the correction factor that is needed to define the convolution product. This is a section of $\Pr_\jac^* \pro^* \Wedgem(\E^\t/\D^\t)^*$, and 
\[
\iota^*\Pr_\jac^*\pro^* [\Omega_\P\otimes  \Wedgem(\E^\t/\D^\t)] = \Omega_{\iota^*\Q}\ .
\]
Again, let $m = \rk \P_0 - \rk \D_0$. Over $\hbar\neq0$,
\[
\omega_\Gamma = \hbar^{-1}\Pr_\jac^*\pro^*\omega + \hbar^{-2}d\hbar\wedge\Pr_\jac^*\theta\ .
\]
Because $\varepsilon_\Q$ is constructed from the restriction of $\omega_\Gamma^m$,  
\[
\varepsilon_\Q = \hbar^{-m}\Pr_\jac^* \pro^*\varepsilon_\P \ ,
\]
so $\varepsilon_\Q$ appears to diverge at $\hbar=0$.

\begin{definition}
A smooth section $a$ of $\Omega_\Q^{1/2}$ is a smooth section of $\Pr_\jac^*\pro^* \Omega_\P^{1/2}$ over $\hbar\neq0$ such that $a\,\varepsilon_\Q^{-1/2}$ extends to a smooth section. 
\end{definition}
This effectively means that as a bundle, $\Omega_\Q$ is actually $\Pr_\jac^*\pro^* [\Omega_\P\otimes  \Wedgem(\E^\t/\D^\t)]$, but the Bott connection and identifications such as $\inv^*\Omega_\Q = \bar\Omega_\Q$ come from the original definition over $\hbar\neq0$.

\begin{definition}
\label{Restriction}
\[
\ev_0 : \Gamma(\Gamma,\Omega_\Q^{1/2}) \to \Gamma(\Gamma\rvert_{\hbar=0},\Omega^{1/2}_{\iota^*\Q})
\]
is given by 
\[
\ev_0(a) := \iota^*(a\,\varepsilon_\Q^{-1/2})\ ,
\]
where the extension to a smooth section should be understood.
\end{definition}
With some abuse of notation, this means that
\[
\ev_0(a) = \lim_{\hbar\to0} \hbar^{m/2}a\,\varepsilon_\P^{-1/2}\ .
\]

These definitions extend trivially to the tensor product with the prequantization line bundle (which is trivial, but has a nontrivial connection).

\begin{remark}
If $m$ is odd, then the power $\hbar^{m/2}$ looks like a problem for $\hbar<0$. In practice, it is not. For $m\neq0$, the polarization is complex and the condition of positivity (Def.~\ref{Positive}) is nontrivial. If $\P$ is positive everywhere, then $\Q$ is positive only for $\hbar\geq0$. This means that we should discard the $\hbar<0$ part of $\Gamma$ and never have to worry about taking a square root of a negative $\hbar$.
\end{remark}

\subsection{Quantization}
Again, let $\P$ be a polarization of the symplectic groupoid $(\Sigma,\omega)$ such that $\rk \D_0$ is locally constant. Let $\tilde\P$ and $\Q$ again be constructed from $\P$ as in the previous section.

\begin{definition}
\label{Reeb}
The \emph{Reeb vector field} $R\in\X^1(\jac)$ on $(\jac,\theta)$ is determined by  the conditions $R\inner\theta=1$ and $R\inner d\theta=0$.
\end{definition}
The first condition implies that $R\neq0$ everywhere. The second condition implies that $T\pro(R)=0$, so the fibers of $\pro$ are the orbits of $R$.
The quantization of $\Hei(M,\Pi)$ that is constructed with $\Gamma$ is quite different depending upon whether these fibers are lines or circles.

\subsubsection{Simply connected fibers}
\label{Simply connected}
There exists a choice of $\pro:\jac\to\Sigma$ with simply connected fibers if and only if the Poisson manifold $(M,\Pi)$ is integrable to a symplectic groupoid and  its periods \cite{c-z} are all $0$.

\begin{definition}
Let $\pro_2 : \jac_2\to\Sigma_2$ denote the restriction of $\pro\times \pro$.
\end{definition}
\begin{thm}
\label{R fibers}
If the fibers of $\pro:\jac\to\Sigma$ are lines and $\Sigma$ is paracompact, then there exist $\theta_0\in\Omega^1(\Sigma)$ and a cocycle $c\in\Ci(\Sigma_2)$ such that $\omega = -d\theta_0$, 
\[
\partial_\Sigma^*\theta_0 =
\pr_{1\,\Sigma}^*\theta_0 - \m_\Sigma^*\theta_0+\pr_{2\,\Sigma}^*\theta_0 = \pro_2^*dc\ ,
\] 
both $c$ and $dc$ vanish along the image of $M$, and $\jac$ is isomorphic to the central extension $\Sigma\times\R$ of $\Sigma$ determined by the cocycle $c$. If $\jac$ is identified with $\Sigma\times\R$, then $\pro$ is the canonical projection, and $\theta = dz+ \pro^*\theta_0$, where $z$ is the coordinate on $\R$.
\end{thm}
\begin{proof}
The groupoid structure of $\jac$ makes $\pro:\jac\to\Sigma$ an affine bundle, but by paracompactness, we can use a locally finite cover of $\Sigma$, local sections,  and a partition of unity to construct a global section $\zeta:\Sigma\to\jac$ of $\pro$. 

Let $\theta_0:=\zeta^*\theta$. By eq.~\eqref{symplectic quotient},
\[
\omega = \zeta^*\pro^*\omega = -\zeta^*d\theta = -d\theta_0\ .
\]
Define $z:\jac\to\R$ to be the (unique) function such that $\zeta^*z=0$ and $R\inner dz=1$. The map
\[
(\pro,z) : \jac\to \Sigma\times\R
\]
is a diffeomorphism.

Consider the difference, $\theta-dz-\pro^*\theta_0$. By construction, $\zeta^*(\theta-dz-\pro^*\theta_0)=0$,
\[
R\inner(\theta-dz-\pro^*\theta_0) = 1-1+0 =0 \ ,
\]
and the Lie derivative is
\[
\Lie_R(\theta-dz-\pro^*\theta_0)=0\ .
\]
As the fibers of $\pro$ are the orbits of $R$, this implies that this difference is $0$ and therefore
\[
\theta = dz+ \pro^*\theta_0\ .
\]

Now define $\tilde c := -\partial_\jac^*z \in \Ci(\jac_2)$. The multiplicativity of $\theta$ means that
\[
0 = \partial_\jac^*\theta = -d\tilde c + \partial_\jac^*\pro^*\theta_0
=-d\tilde c + \pro_2^*\partial_\Sigma^*\theta_0 \ ,
\]
so $\tilde c$ is constant along the fibers of $\pro_2$, therefore there exists $c\in\Ci(\Sigma_2)$ such that $\tilde c = \pro_2^* c$. This $c$ is a cocycle  because $\tilde c$ is a coboundary. 

If $(\pro,z)$ is used to identify $\jac$ with $\Sigma\times\R$, then from this definition of $c$ it is easy to read off the explicit formula for the product on $\Sigma\times\R$.

None of this implies anything about  vanishing of $c$ along $M$. For that, we need to choose the section $\zeta$ more carefully. Firstly, for every $\gamma\in\Sigma$, we can replace $\zeta(\gamma)$ with the midpoint between $\zeta(\gamma)$ and $[\zeta(\gamma^{-1})]^{-1}$. In this way, we can always choose the section so that $\zeta(\gamma^{-1})=[\zeta(\gamma)]^{-1}$. This implies that $\inv^*\theta_0=-\theta_0$ and therefore $\unit^*\theta_0=0$. This means that $\theta_0$ is at least normal to $M$. Moreover, $\zeta\circ\unit_\Sigma = \unit$, so $z$ vanishes along the unit submanifold, $M$, and $c$ vanishes along the image of $M$ in $\Sigma_2$. 

Any change in $\zeta$ can be described as adding a function $f\in\Ci(\Sigma)$ to $\zeta$. This changes $\theta_0$ to $\theta_0+df$. The function $f$ can be easily chosen to vanish along $M$ and have a normal derivative that cancels $\theta_0$ along $M$.
With this choice, $\theta_0$ will vanish along the unit submanifold, and so $dc$ vanishes along the image of $M$.
\end{proof}
So, let $\jac=\Sigma\times\R$ be given explicitly in this way. 
As a manifold,
\[
\Gamma = \EE(\Sigma,M)\times_\hbar \EE(\R,\{0\},0) = \EE(\Sigma,M) \times \R\ .
\]
In this double explosion, the $\R$ factor is rescaled by $\hbar^2$, so as a groupoid, $\Gamma$ is the central extension of $\EE(\Sigma,M)$ determined by $\hbar^{-2}\Pr_{\Sigma_2}^*c$, which is actually a smooth cocycle of $\EE(\Sigma,M)$ because of the vanishing of $c$ and $dc$ along $M$.

Because the symplectic form $\omega_\Gamma$ is exact, $(\Gamma,\omega_\Gamma)$ can be prequantized with the trivial line bundle, $\co\times\Gamma$. The connection is given by a symplectic potential, and the construction of $\Gamma$  gives a canonical symplectic potential,
\[
\EA_\jac \theta = \hbar\,dz + 2z\,d\hbar+\EA_\Sigma\theta_0 \ .
\]
This is multiplicative, but as we shall see, it is better not to have a $dz$ term, so let's instead use the symplectic potential,
\beq
\label{corrected connection}
\theta_\Gamma :=\EA_\jac \theta - d(\hbar z) = z\,d\hbar+\EA_\Sigma\theta_0 \ .
\eeq
This gives the connection $\nabla = d+i\theta_\Gamma$ on $\co\times\Gamma$.

The third ingredient of the prequantization is a cocycle, $\sigma$. This  is a parallel section of the trivial line bundle over $\Gamma_2$ with connection $d-i\partial_\Gamma^*\theta_\Gamma$. It can also be normalized so that $(\unit,\unit)^*\sigma = 1$. With this choice of $\theta_\Gamma$, we have
\[
\partial_\Gamma^*\theta_\Gamma = \partial_\Gamma^*\EA_\jac\theta - \partial_\Gamma^*d(\hbar z)
= -d(\hbar\partial_\jac^*z) = d(\hbar^{-1}\Pr_{\Sigma_2}^*c) \ ,
\]
therefore
\[
\sigma = e^{i \hbar^{-1}\Pr_{\Sigma_2}^*(c)} \ .
\]

Notably, the restriction of this prequantization to the subgroupoid at some fixed $\hbar\neq0$ gives precisely the prequantization of $\Sigma$ with symplectic potential $\theta_0$ (and lifted trivially to $\Sigma\times\R$).

The polarization is simply $\Q = \tilde\P\times T_\co\R$, i.e., it is $\tilde\P$ along $\EE(\Sigma,M)$ and everything along $\R$ (the $z$-direction). The convolution algebra consists of $\Q$-polarized sections of $\Omega^{1/2}_\Q$ with the connection modified by $\theta_\Gamma$. In particular, a polarized section must be parallel in the $z$-direction, but because there is no $dz$ term in $\theta_\Gamma$, it is simply constant. This means that  a polarized section is just the pullback (by $\EE\pro$) of a $\tilde\P$-polarized section over $\EE(\Sigma,M)$.

The convolution product is supposed to be defined by integrating over a transversal to the real part of $\Q$ in each $\EE\t$ fiber. These transversals can be chosen at constant $z$. The result is that $z$ plays no role in the convolution product. We can simply work with $\tilde\P$-polarized sections and convolution over $\EE(\Sigma,M)$. Heuristically,
\[
\cs_\Q(\Gamma,\sigma) = \cs_{\tilde\P}(\EE(\Sigma,M),\sigma)\ .
\]

The subgroupoid of $\EE(\Sigma,M)$ over some fixed, nonzero value of $\hbar$ is a copy of $\Sigma$. The restriction of $\tilde\P$ to this is simply $\P$. The restriction of the prequantization is the line bundle with connection $d+i\hbar\theta_0$ and cocycle $\sigma(\hbar) = e^{i\hbar^{-1}c}$. 
The restriction of a $\tilde\P$-polarized section, $a$, is a $\P$-polarized section $\ev_\hbar a \in \Gamma(\Sigma,\Omega_\P^{1/2})$.
For two polarized sections, $a$ and $b$, $\ev_\hbar(a*b)=(\ev_\hbar a)*(\ev_\hbar b)$. In other words, restriction is a homomorphism 
\[
\ev_\hbar:\cs_\Q(\Gamma,\sigma)\to\cs_\P(\Sigma,\sigma(\hbar))\ .
\]

The subgroupoid of $\EE(\Sigma,M)$ over $\hbar=0$ is $T^*M$ (a bundle of abelian groups). The restriction $\iota^*\tilde\P$ of $\tilde\P$ is a combination of $\D_0^\co$ in the horizontal directions and its annihilator in the vertical directions (thus giving a Lagrangian distribution). The restriction of the prequantization is the trivial line bundle  with connection determined by the Liouville form, $\theta_{T^*M}$, and the trivial cocycle. It appears from examples  \cite[\S8.1]{haw10} that   this twisted polarized convolution algebra, $\cs_{\iota^*\tilde\P}(T^*M,\theta_{T^*M})$,  is isomorphic to $\C_0(M)$.

Again, for a $\tilde\P$-polarized section, $\ev_0(a)\in\Gamma(T^*M,\Omega^{1/2}_{\iota^*\tilde\P})$ is an $\iota^*\tilde\P$-polarized section, and $\ev_0 : \cs_\Q(\Gamma,\sigma) \to \cs_{\iota^*\tilde\P}(T^*M,\theta_{T^*M})$ is a homomorphism, which should be thought of as restriction to $\hbar=0$. 

The idea is that the algebra $\cs_\Q(\Gamma,\sigma)$ should be the algebra of $\C_0$-sections of a continuous field of $\cs$-algebras over $\R$, and $\ev_\hbar$ should be  the evaluation map from  sections to the fiber at $\hbar$.

\subsubsection{Multiply connected fibers}
\label{Multiply}
Now suppose that the fibers of $\pro:\jac\to\Sigma$ are circles. In this case, Bohr-Sommerfeld conditions necessarily come into play, because the polarization $\Q$ includes  the directions tangent to the fibers of $\EE\pro$.

Recall (Sec.~\ref{Prequantization}) that $(\Gamma,\omega_\Gamma)$ can be prequantized by the trivial bundle $\co\times\Gamma$, the connection $\nabla=d+i\EA_\jac\theta$, and the trivial cocycle $1$. The algebra is supposed to be constructed from $\Q$-polarized sections of $\Omega_\Q^{1/2}$, with the connection modified by $\EA_\jac\theta$.  In particular, a $\Q$-polarized section must  be parallel along the fibers of $\EE\pro$. 

Consider what happens over some point $(\gamma,\hbar)\in\Sigma\times\R_*\subset \EE(\Sigma,M)$. If the fiber $(\EE\pro)^{-1}(\gamma,\hbar) = \pro^{-1}(\gamma)\times\{\hbar\}$ is a circle, then parallel transport around this gives a phase factor (holonomy) of
\[
e^{-i\oint_{(\EE\pro)^{-1}(\gamma,\hbar)} \EA_\jac \theta} = e^{-i\hbar^{-1}\oint_{\pro^{-1}(\gamma)}\theta} .
\]
If this phase is not $1$, then a polarized section must necessarily vanish at any point of $(\EE\pro)^{-1}(\gamma,\hbar)$. This is an example of a Bohr-Sommerfeld condition. 

\begin{thm}
\label{Period}
If the fibers of $\pro:\jac\onto\Sigma$ are circles, then the quantity
\beq
\label{period}
\oint_{\pro^{-1}(\gamma)} \theta 
\eeq
is a locally constant function of $\gamma\in\Sigma$.
\end{thm}
\begin{proof}
Consider a curve $C\subset\Sigma$. The difference of \eqref{period} between the endpoints of $C$ is equal to 
\[
\int_{\pro^{-1}(C)} d\theta = -\int_{\pro^{-1}(C)} \pro^*\omega \ .
\]
However, $\pro^*\omega$ is degenerate in the direction of the fibers of $\pro$, which is tangent to $\pro^{-1}(C)$, therefore this integral is $0$.
\end{proof}

For example, let $(L,\nabla,\sigma)$ be a prequantization of a symplectic groupoid $(\Sigma,\omega)$ over $M$. Let $\Sigma^\sigma \subset L^*$ be the unit circle bundle --- or equivalently, the principal $\T$-bundle associated to $L$. Let $\pro:\Sigma^\sigma \to \Sigma$ be the bundle projection. The complex conjugate cocycle $\bar\sigma$ defines a multiplication that makes $L^*$ a category and $\Sigma^\sigma$ a subgroupoid.   Let $\theta\in\Omega^1(\Sigma^\sigma)$ be the connection $1$-form; this is normalized so that for any $\gamma\in\Sigma$,
\beq
\label{normal period}
\oint_{\pro^{-1}(\gamma)} \theta = 2\pi \ .
\eeq

Conversely, if the fibers of $\pro :\jac\to\Sigma$ are circles and the contact form is normalized to satisfy \eqref{normal period}, then $\jac$ is isomorphic to $\Sigma^\sigma$ for some prequantization of $(\Sigma,\omega)$. More generally, if $M$ is connected, then by Theorem~\ref{Period} the symplectic structure can always be rescaled to satisfy \eqref{normal period}. This is equivalent to just rescaling $\hbar$ in $\EE(\jac,M,T^*M)$, so we can consider the case that $\jac=\Sigma^\sigma$ without loss of generality.

So, let $\Gamma=\EE(\Sigma^\sigma,M,T^*M)$.  In this case, the fibers of $\EE\pro$ are circles where $\hbar\neq0$ and lines  where $\hbar=0$. For $\hbar\neq0$, the holonomy around the fiber $(\EE\pro)^{-1}(\gamma,\hbar)$ is $e^{-2\pi i /\hbar}$, which is only trivial if $\hbar^{-1}$ is an integer. If there are no other Bohr-Sommerfeld conditions, then this means the Bohr-Sommerfeld subvariety is the subset $\Gamma_\BS\subset\Gamma$ of points where $\hbar=0$ or $\hbar^{-1}\in\Z$.

Let's consider the component  $\Gamma_k\subset\Gamma_\BS$ at $\hbar=k^{-1}$. As a groupoid, this is isomorphic to $\Sigma^\sigma$. The prequantization connection restricts to the connection $d+ik\theta$. The polarization, $\Q$, restricts to $(T\pro)^{-1}\P$.
The Reeb vector field (Def.~\ref{Reeb}) $R\in\X^1(\Sigma^\sigma)$ associated to the contact structure $\theta$ also generates the action of $\T$ on $\Sigma^\sigma$.
Because $T\pro(R)=0$, $R$ is a section of the polarization, so a polarized function $f\in\Ci(\Gamma_k)$ must in particular satisfy
\[
0 = R\inner df + ikR\inner\theta f = R\inner df +ikf\ .
\]
Therefore, $f$ is equivalent to a section of the line bundle $L^{\otimes k}$ over $\Sigma$.

In this way, a polarized section $a\in\Gamma(\Gamma_\BS,\Omega_\Q^{1/2})$, restricts at $\hbar=k^{-1}$ to a polarized section of $L^{\otimes k}\otimes \Omega_\P^{1/2}$. Restriction is a homomorphism  
\[
\ev_{k^{-1}}:\cs_\Q(\Gamma,\EA_\jac\theta)\to\cs(\Sigma,\sigma^k)\ .
\]

The subgroupoid $\Gamma\rvert_{\hbar=0}$ is $T^*M\times\R$, the central extension determined by $\Pi$. The restriction $\iota^*\EA_\jac\theta$ is just the  Liouville form $\theta_{T^*M}$, lifted to $T^*M\times\R$. Again, a polarized section is constant along the $z$-direction here 

Let $a\in\Gamma(\Gamma_\BS,\Omega_\Q^{1/2})$ be a polarized section. The behavior of $a$ is very simple along each $\EE\pro$ fiber. At $\hbar=k^{-1}$, the phase of $a$ goes around $k$ times as we go around a fiber of $\EE\pro$. As $k\to\infty$ and $\hbar\to0$, these fibers expand out to lines faster than the phase changes. The result is that $\ev_0(a)$ is independent of $z$ and is a polarized section. Again, 
\[
\ev_0:\cs_\Q(\Gamma,\EA_\jac\theta) \to \cs_{\iota^*\tilde\P}(T^*M,\theta_{T^*M})
\]
is a homomorphism.

This shows that in constructing a strict deformation quantization by geometric quantization, it is not necessary to put the set, $I$, of $\hbar$-values in by hand. Geometric quantization of a Heisenberg-Poisson manifold will take integrality conditions into account precisely where they are needed.

\section{Examples}
\label{Examples}
\subsection{Constant Poisson structure}
Let $V$ be a finite dimensional, real vector space and $\Pi\in\Wedge^2V$. If we identify $\Pi$ with a constant bivector on $V$, then $(V,\Pi)$ is a Poisson manifold.
The $\s$-simply-connected symplectic groupoid $(\Sigma,\omega)$ integrating $(V,\Pi)$ can be constructed as the semidirect product $V^*\ltimes V$, where $V^*$ acts on $V$ by translations, via the linear map $\#_\Pi:V^*\to V$. 

Let $x^i$ be coordinates on $V$ and $y_i$ coordinates on $V^*$, although I will mainly use index-free notation. In particular, $\#_\Pi y$ means $\Pi^{ji}y_j$ and $\Pi(y,y')=\Pi^{ij}y_iy'_j$. 

Identifying $\Sigma$ with $V\times V^*$, the unit and inverse are
\[
\unit_\Sigma(x)=(x,0) \ , \qquad
\inv_\Sigma(x,y)=(x,-y)\ .
\]
The source and target maps are
\[
\s_\Sigma(x,y) = x - \tfrac12 \#_\Pi y \ , \qquad
\t_\Sigma(x,y) = x + \tfrac12 \#_\Pi y \ .
\]
Identifying the set of composable pairs, $\Sigma_2$ with $V\times V^*\times V^*$, the other structure maps are
\begin{gather*}
\pr_{1\, \Sigma}(x,y,y') = (x + \tfrac12 \#_\Pi y',y) \\
\pr_{2\, \Sigma}(x,y,y') = (x - \tfrac12 \#_\Pi y,y') \\
\m_\Sigma(x,y,y') = (x,y+y') .
\end{gather*}
The symplectic form is $\omega = dx^i\wedge dy_i$. 

The $\s$-simply-connected contact groupoid $(\jac,\theta)$ integrating $(V,\Pi)$ is an extension of $\Sigma$ by $\R$. I shall identify $\jac$ with $V\times V^*\times \R$, and so we need one more coordinate, $z$. The unit, inverse, source, and target maps for $\jac$ are
\begin{gather*}
\unit(x)=(x,0,0), \qquad
\inv(x,y,z) = (x,-y,-z) ,\\
 \s(x,y,z) = x - \tfrac12 \#_\Pi y \ , \qquad
\t(x,y,z) = x + \tfrac12 \#_\Pi y \ .
\end{gather*}
Identifying $\jac_2$ with $V\times V^*\times V^* \times \R^2$, the remaining structure maps are
\begin{gather*}
\pr_1(x,y,y',z,z') = (x + \tfrac12 \#_\Pi y',y,z) \\
\pr_2(x,y,y',z,z') = (x - \tfrac12 \#_\Pi y,y',z') \\
\m(x,y,y',z,z') = (x,y+y',z+z'-\tfrac12\Pi[y,y']) \ .
\end{gather*}
The contact form is
\[
\theta = dz + y_idx^i\ .
\]

In this form, $(\jac,V,T^*V)= (V\times V^*\times\R,V,V^*)$ and $(\jac_2,V,T^*V\oplus T^*V) = (V\times V^*\times V^*\times \R^2,V,V^*\oplus V^*)$ are  explosive charts, so the structure maps for $\Gamma$ can be written down immediately:
\begin{gather*}
\EE\unit(x,\hbar)=(x,0,0,\hbar), \qquad
\EE\inv(x,y,z,\hbar) = (x,-y,-z,\hbar) ,\\
 \EE\s(x,y,z,\hbar) = (x - \tfrac12\hbar\, \#_\Pi y,\hbar) \ , \qquad
\EE\t(x,y,z,\hbar) = (x + \tfrac12\hbar\, \#_\Pi y,\hbar) \ , \\
\EE\pr_1(x,y,y',z,z',\hbar) = (x + \tfrac12\hbar\, \#_\Pi y',y,z,\hbar) \\
\EE\pr_2(x,y,y',z,z',\hbar) = (x - \tfrac12\hbar\, \#_\Pi y,y',z',\hbar) \\
\EE\m(x,y,y',z,z',\hbar) = (x,y+y',z+z'-\tfrac12\Pi[y,y'],\hbar) .
\end{gather*}
\begin{remark}
This is compatible with a grading where $\hbar$, $x$, $y$ and $z$ have degrees $-1$, $0$, $1$, and $2$, respectively. The factors of $\hbar$ come in where the degree needs to be corrected.
\end{remark}

The exploded contact form is 
\[
\EA_\jac \theta = y_i\,dx^i+\hbar\,dz+2z\,d\hbar \ ,
\]
which gives the symplectic form on $\Gamma$,
\[
\omega_\Gamma = -d(\EA_\jac \theta) = dx^i\wedge dy_i + d\hbar\wedge dz\ .
\]

\begin{remark}
This fits into the general description in Section~\ref{Simply connected}. The symplectic potential on $\Sigma$ is $\theta_0 = y_idx^i$; $\jac$ is the central extension of $\Sigma$ determined by the cocycle 
\[
c(x,y,y') = -\tfrac12\Pi(y,y') \ .
\]
The natural projection is $\Pr_{\Sigma_2}(x,y,y',\hbar) = (x,\hbar y,\hbar y')$, so 
\[
(\Pr_{\Sigma_2}^*c)(x,y,y',\hbar) = -\tfrac12\hbar^2\Pi(y,y') \ ,
\] 
and $\Gamma$ is the central extension of $\EE(\Sigma,V)$ determined by  the cocycle
\[
\hbar^{-2}(\Pr_{\Sigma_2}^*c)(x,y,y',\hbar) = -\tfrac12\Pi(y,y')\ .
\]
\end{remark}

\begin{remark}
If we view $V^*$ as an Abelian Lie algebra, then $\Pi$ is a cocycle and the central extension is a Heisenberg Lie algebra. The dual of this Heisenberg Lie algebra is the Heisenberg-Poisson manifold, $\Hei(V,\Pi)$. The symplectic groupoid $(\Gamma,\omega_\Gamma)$ is the cotangent bundle of the corresponding Heisenberg group.
\end{remark}

\subsubsection{Polarizations and quantization}
Now, let $P\subset V_\co \oplus V^*_\co$ be Lagrangian, multiplicative, Hermitian, and positive (in the sense of Definition~\ref{Positive}). This defines a \emph{constant} polarization $\P = P\times\Sigma$ of $\Sigma$. There are 3 simple special cases, and the general case is a product of these 3 types.

\emph{First case:} 
If $P=V^*_\co\subset V_\co\oplus V^*_\co$, then (to be Lagrangian) we must have $\Pi=0$, and so $\Gamma$ is a trivial bundle of abelian groups over $\Hei V$. This polarization is just  (the complexification of) the kernel foliation of the fibration $\s=\t:\Gamma \to \Hei V$.
The half-form bundle is  the trivial line bundle.  

The symplectic potential
\[
\theta_\Gamma = \EA_\jac\theta - d(\hbar z) = y_idx^i+ z\,d\hbar
\]
is adapted to this polarization. Because $\Pi=0$, $\theta_\Gamma$ is multiplicative, therefore there is a prequantization of $\Gamma$ by the trivial line bundle with connection given by $\theta_\Gamma$ and the trivial cocycle.

Because this is a simple real polarization and the cocycle is trivial, the twisted polarized convolution \cs-algebra is just the \cs-algebra of the quotient groupoid, which is  the manifold $\Hei V = V\times \R$:
\[
\cs_\Q(\Gamma,\theta_\Gamma) = \C_0(V\times \R) \ .
\]

The part at $\hbar=0$ is similarly trivial,
\[
\cs_{\iota^*\Q}(\Gamma\rvert_{\hbar=0},y_idx^i) = \C_0(V)\ .
\]
Because $m = \rk \P_0 - \rk \D_0 = 0$, the restriction map
\[
\ev_0 : \C_0(V\times\R) \to \C_0(V)
\]
is just evaluation at $0\in\R$.

\emph{Second case:}
If $P = V_\co\subset V_\co\oplus V^*_\co$, then the half-form bundle has fiber $\Wedgem V$. This choice of polarization is consistent with any constant Poisson structure, $\Pi$. 

In this case, the symplectic potential 
\[
\theta_\Gamma = \EA_\jac \theta - d(y_ix^i+\hbar z) = -x^i dy_i + z\,d\hbar
\]
is adapted to this polarization, because it contains no $dy_i$ or $dz$. The term $\EA_\jac\theta$ is automatically multiplicative and 
\[
\partial_\Gamma^*(y_ix^i+\hbar z) = -\tfrac12\hbar \Pi(y, y')\ ,
\]
so $\partial_\Gamma^*\theta_\Gamma = d\left[\frac12\hbar \Pi(y, y')\right]$. The prequantization of $\Gamma$ is the trivial line bundle with connection $d+i\theta_\Gamma$ and a cocycle, $\sigma$, which is a covariantly constant section of the trivial line bundle over $\Gamma_2$ with connection $d-i\partial_\Gamma^*\theta_\Gamma$, therefore
\[
\sigma(x,y,y',z,z',\hbar) = e^{\frac{i}2 \hbar\Pi(y, y')}\ .
\]

This is another simple real polarization. The quotient groupoid is $V^*\times\R$, a trivial bundle of Abelian groups over $\R$. The cocycle $\sigma$ is just the pullback of a cocycle $\sigma_0$ on the quotient groupoid,
\[
\sigma(x,y,y',z,z',\hbar) = \sigma_0(y,y',\hbar)\ .
\]
The twisted, polarized convolution \cs-algebra is
\[
\cs_\Q(\Gamma,\sigma) = \cs(V^*\times\R,\sigma_0)\ .
\]

At $\hbar=0$, $\sigma_0$ is trivial, so
\[
\cs_{\iota^*\Q}(\Gamma\rvert_{\hbar=0},\sigma_0) = \cs(V^*) \ .
\]
The Fourier transform gives an isomorphism $\cs(V^*)\cong\C_0(V)$. Because $m=0$ again,
the evaluation map
\[
\ev_0 : \cs(V^*\times\R,\sigma_0) \to \cs(V^*) \cong \C_0(V)
\]
is simply the restriction to $\hbar=0$.

The result is that $\cs(V^*\times\R,\sigma_0)$ is the algebra of sections of a continuous field of \cs-algebras over $\R$. The fiber over $0$ is $\C_0(V)$, and the fiber over any other $\hbar$ is a twisted convolution algebra of $V^*$, which is isomorphic to 
\[
\K[L^2(\R^{\frac12\rk\Pi})]\otimes \C_0(\R^{\dim\ker\Pi})\ .
\]

\emph{Third case:}
If $P\cap\bar P =0$, then the Poisson structure is symplectic and the polarization is determined by a constant \Kahler\ structure on $V$. 
In this case $m = \dim P = \frac12\dim V$.
The half-form bundle has fiber a square root of $\Wedge^{2m} V = \Wedgem V$.

The most convenient symplectic potential in this case is
\[
\theta_\Gamma = \EA_\jac\theta - d(\hbar z) \ .
\]
The coboundary of $\hbar z$ is
\[
\partial_\Gamma^*(\hbar z) = \tfrac12\hbar \Pi(y,y')
\]
so the cocycle in this prequantization is
\[
\sigma(x,y,y',z,z',\hbar) = e^{-\frac{i}2 \hbar\Pi(y, y')}\ .
\]

Because the polarization $\P$ is nontrivially positive, the polarization $\Q$ of $\Gamma$ is positive for $\hbar\geq0$, but not for $\hbar<0$, so let $\Gamma_{\geq0}$ be the $\hbar\geq0$ part. The convolution algebra consists of $\Q$-polarized $\sqrt{\Wedgem V}$-valued functions over $\Gamma_{\geq0}$ with the connection $d+i\theta_\Gamma$.

As in the general case of Section~\ref{Simply connected}, these polarized functions are constant in the $z$-direction. They are thus equivalent to polarized functions over the $\hbar\geq0$ part of $\EE(\Sigma,V)$ and we can ignore the $z$ coordinate.

The correction factor $\varepsilon_\Q \in\Ci(\Gamma,\Wedgem V)$ is defined from the restriction of $\omega$ to $T^\t\Gamma$, which is $-\hbar\Pi$ thought of as a $2$-form on $V^*$, so $\varepsilon_\Q = \hbar^m\varepsilon$, where 
\[
\varepsilon = \frac{1}{m!} \left(\frac{\Pi}{2\pi}\right)^m \in \Wedge^{2m} V\ .
\]
\begin{remark}
Rather confusingly, it now appears that $\varepsilon_\Q$ vanishes at $\hbar=0$, whereas in Section~\ref{Half-forms}, it appeared that $\varepsilon_\Q$ diverges as $\hbar\to0$. The difference comes from the way that $\Gamma$ is presented.

Here, $\Gamma \cong V\times V^*\times\R$, where the groupoid structure comes from $V^*$ acting on $V$ in a way that depends upon $\hbar\in\R$. Elements of $\Wedgem V$ are regarded as volume forms on $V^*$. 

On the other hand, the $\hbar\neq0$ part of $\Gamma$ can be identified with $V\times V^*\times\R_*$, where the action of $V^*$ on $V$ is independent of $\hbar\in\R_*$. To transform between these pictures, there is a rescaling of $V^*$ by $\hbar$, which transforms $\Wedgem V$ by a factor of $\hbar^{-m}$. This is the difference between $\varepsilon_\Q$ vanishing and diverging.
\end{remark}

Now, let $a,b\in  \Ci(\Gamma_{\geq0},\sqrt{\Wedgem V})$ be two polarized functions. Because, $P$ has no real directions, the $z$-direction is the only real direction of $\Q$ (at $\hbar\neq0$). This means that the convolution product $a*b$ is defined by integrating
\[
\EE\pr_1^*a \cdot \EE\pr_2^*b\cdot \sigma \cdot \EE\pr_1^*(\sqrt{\varepsilon_\Q})
\]
over the fibers of $\EE\m$ in $\EE(\Sigma,V)$. This leads to the formula, \begin{multline}
\label{product}
(a*b)(x,y,\hbar) =\\ \hbar^{m/2}\varepsilon^{1/2}\int_{y'+y''=y} a(x+\tfrac12\hbar\#_\Pi y'',y',\hbar) b(x-\tfrac12\hbar\#_\Pi y',y'',\hbar) e^{-\frac{i}2 \hbar\Pi(y', y'')} 
\end{multline}

At $\hbar=0$, the polarization $\iota^*\Q$ points in the $y$ and $z$-directions. The quotient groupoid (leaf space of $\iota^*\Q$) is just $V$, as a trivial groupoid, and the restriction of the cocycle $\sigma$ is trivial, therefore
\[
\cs_{\iota^*\Q}(\Gamma\rvert_{\hbar=0},\sigma) = \C_0(V) \ .
\]

Again, let $a \in \Ci(\Gamma_{\geq0},\sqrt{\Wedgem V})$ be a $\Q$-polarized function. The restriction to $\hbar=0$ is defined (Def.~\ref{Restriction}) as 
\[
\ev_0(a)(x,y) = \lim_{\hbar\to0^+} \hbar^{-m/2}\varepsilon^{-1/2} a(x,y,\hbar)\ .
\]
The existence of this limit is a condition on $a$, but the $\Q$-polarization of $a$ implies that this is independent of $y$, so it is $\iota^*\Q$-polarized. Indeed,
\[
\ev_0 : \cs_\Q(\Gamma,\sigma) \to \C_0(V) 
\]
is a ${}^*$-homomorphism.

The simplest example of this is not actually an element of the algebra. There is no unit in this convolution algebra, but there is a unit multiplier given by a smooth kernel that merely fails to fall off in the $x$-directions. This is
\[
K(x,y,\hbar) =  e^{-\frac14\hbar\,\norm{y}^2} \hbar^{m/2}\varepsilon^{1/2}\ ,
\]
where $\norm{y}$ is the norm in the \Kahler\ metric on $V^*$. Note that this is translation invariant in the sense that it does not depend on $x$. The shape of $K$ is easily determined from translation invariance and polarization. The normalization is determined by the condition that $K*K=K$.

The restriction to $\hbar=0$ is
\[
\ev_0(K) = \lim_{\hbar\to0^+} e^{-\frac14\hbar\,\norm{y}^2} = 1\ ,
\]
as it should be. The actual elements of the algebra behave similarly as $\hbar\to0$.

\begin{remark}
In particular, this justifies the power of $2\pi$ included in Definition~\ref{Epsilon factor} --- it normalizes a Gaussian integral in $K*K$. There have to be powers of $2\pi$ somewhere, and this is the most uniform way of distributing them.
\end{remark}

\subsection{The symplectic sphere}
\label{Sphere}
Consider the symplectic manifold $(S^2,\epsilon)$, where  $\epsilon\in\Omega^2(S^2)$ is half the volume form of the unit sphere. With this normalization, the symplectic volume is $2\pi$.

Everything that I will do here is equivariant with respect to the $\SU(2)$ action on $S^2$. Any (complex) equivariant vector bundle is isomorphic to a direct sum of equivariant line bundles. Equivariant line bundles are classified by irreducible representations of $\T=\U(1)$ --- and thus by integers.
\begin{definition}
For any $k\in\Z$, let $L^k$ denote the equivariant line bundle over $S^2$ associated to the $\T$-representation, $z\mapsto z^{-k}$ and equipped with an equivariant connection and fiberwise, Hermitian inner product.

Let $(j)$ denote the $2j+1$-dimensional Hilbert space carrying the spin $j$ representation of $\SU(2)$. 
\end{definition}

A space of sections of $L^k$ is a direct sum, 
\[
(\abs{\tfrac{k}{2}}) \oplus (\abs{\tfrac{k}{2}}+1) \oplus (\abs{\tfrac{k}{2}}+2) \oplus \dots\ .
\]
For the space of polynomial sections, this is the algebraic direct sum, for the space of square-integrable sections it is the Hilbert space direct sum, and for the space of smooth sections it is the space of sequences whose norms fall off faster than any power.

\subsubsection{Standard geometric quantization}
First, for any integer $k$, I will consider the construction of a Hilbert space by  geometric quantization of $(S^2,k\epsilon)$ with the \Kahler\ polarization. The line bundle $L^k$ has curvature $k\epsilon$, so $L^k$ will be the prequantization line bundle.

The complex structure on $S^2$ decomposes the complexified tangent bundle into a direct sum $T_\co S^2 = \bar\F\oplus\F$ of the holomorphic and antiholomorphic tangent bundles; these are just the line bundles $\F \cong L^{-2}$ and $\bar\F \cong L^2$. The \Kahler\ polarization of $S^2$ is defined by the antiholomorphic tangent bundle, $\F$. 

To apply geometric quantization, we need a half-form bundle \cite{sni,woo}. In general, that is a square root of the maximum exterior power of the annihilator $\F^\perp \subset T_\co^* S^2$. In this case, $\F^\perp \cong L^{-2}$ is a line bundle, so it is its own maximum exterior power, and its unique equivariant square root is  $\sqrt{\F^\perp}\cong L^{-1}$. The tensor product of the prequantization bundle with the half-form bundle is $L^k\otimes\sqrt{\F^\perp} \cong L^{k-1}$.

These are isomorphisms in the category of equivariant vector bundles, but this does not take into account the inner products. For $x\in S^2$, if $\alpha,\beta\in\F^\perp_x$ then the natural Hermitian inner product is $-i\bar\alpha\wedge\beta \in \Wedge^2T^*_x S^2$. On the other hand, the inner product in $L^{-2}$ is $\co$-valued.

Let $\psi,\varphi \in \Gamma(S^2,L^k\otimes\sqrt{\F^\perp})$. 
To construct a Hilbert space, we need a $2$-form-valued inner product, but the intrinsic inner product $\bar\psi\varphi$ is $\sqrt{\Wedge^2T^*_\co S^2}$-valued. The natural correction factor, determined by the symplectic form $k\epsilon$, is $\sqrt{k\epsilon/2\pi}$. The Hilbert space inner product is therefore
\[
\langle\psi\vert\varphi\rangle := \int_{S^2} \bar\psi \varphi  \sqrt{k\epsilon/2\pi} \  .
\]

Now, define the Hilbert space $\Hi_k$ to be the space of holomorphic sections of $L^{k}\otimes\sqrt{\F^\perp}$. That is the kernel of the antiholomorphic derivative map, $\bar\partial$, which maps  sections of $L^{k}\otimes\sqrt{\F^\perp}\cong L^{k-1}$ to sections of $\F^*\otimes L^{k}\otimes\sqrt{\F^\perp} \cong L^{k+1}$, so 
\begin{align*}
\bar\partial : &\Gamma(S^2,L^{k-1}) \to \Gamma(S^2, L^{k+1}) \\
& (\abs{\tfrac{k-1}2}) \oplus (\abs{\tfrac{k-1}2}+1)\oplus \dots 
\to (\abs{\tfrac{k+1}2}) \oplus (\abs{\tfrac{k+1}2}+1)\oplus \dots\ .
\end{align*}
Since $\bar\partial$ is equivariant, this shows that if $k>0$, then $\Hi_k \supseteq (\tfrac{k-1}{2})$. In fact $\Hi_k = (\tfrac{k-1}{2})$ for $k>0$ and $\Hi_k=0$ otherwise. Note that $\dim \Hi_k = k$.

This standard geometric quantization construction gives the quantization of $S^2$ with symplectic form $k\epsilon$ as the algebra, $\Li(\Hi_k)$, of operators on $\Hi_k$, i.e., $k\times k$ matrices.

\begin{definition}
\label{Coherent}
Let $\psi_o \in \Hi_k \cong (\tfrac{k-1}2)$ be a normalized highest weight vector. Let $o=(0,0,1)\in S^2$ be the north pole. The \emph{coherent states}  map (called the \emph{covariant symbol} by Berezin \cite{ber}) is the linear map 
\[
I_k : \Li(\Hi_k) \to \C(S^2) .
\]
such that 
\[
I_k(a)(go) = \langle g\psi_o\rvert a \lvert g\psi_o\rangle ,
\]
for any $a\in \Li(\Hi_k)$ and $g\in\SU(2)$.
\end{definition}
\begin{remark}
This is well defined, because $\SU(2)$ acts transitively on $S^2$, and $\psi_o$ is invariant up to a phase under the subgroup $\T\subset\SU(2)$ that leaves $o$ fixed. The map $I_k$  is \emph{not} a homomorphism but provides a way of comparing elements of this matrix algebra with functions on $S^2$.
\end{remark}

\subsubsection{The groupoid approach}
The algebra $\Li(\Hi_k)$ can also be constructed directly, using the symplectic groupoid $(\Sigma,\omega)$, where $\Sigma = \Pair(S^2)\cong S^2\times S^2$ and $\omega = k(\t^*\epsilon - \s^*\epsilon)$. 
I use  $\boxplus$ and $\boxtimes$ to denote the outer direct sum and tensor product,  each of which takes two vector bundles over $S^2$ and gives a vector bundle over $S^2\times S^2$.

The prequantization line bundle of $\Sigma$ is $L^k\boxtimes L^{-k}$. The cocycle is trivial.

The polarization $\P\subset T_\co \Sigma$ is 
\[
\P := \F\boxplus \bar \F ,
\]
and $T^\t \Sigma = 0 \boxplus T S^2$, so
\[
T^\t_\co\Sigma/\P^\t = 0 \boxplus (T_\co S^2/\bar\F)
\]
and
\[
T^\s_\co\Sigma/\P^\s = (T_\co S^2/\F) \boxplus 0 .
\]
The half-form bundle $\Omega_\P^{1/2}$ is a square root of $\Omega_\P$, which is the dual of the maximum exterior power of 
\beq
\label{normal}
T^\t_\co\Sigma/\P^\t \oplus T^\s_\co\Sigma/\P^\s  
= (T_\co S^2/\F) \boxplus (T_\co S^2/\bar\F) .
\eeq
The dual of \eqref{normal} is $\F^\perp\boxplus \bar\F^\perp$ and $\rk \F^\perp = 1$, so
\[
\Omega_\P = \F^\perp\boxtimes \bar\F^\perp 
\]
and finally 
\beq
\label{half-form S2}
\Omega^{1/2}_\P =\sqrt{ \F^\perp} \boxtimes \sqrt{\bar\F^\perp} 
\cong L^{-1}\boxtimes L^1 .
\eeq

Note that because the polarization is totally complex, the real part is $\D=0$, and $\E=T_\co\Sigma$. 

So, the algebra, $A_k = \cs_{\F\boxplus\bar\F}(\Pair S^2,L^k\boxtimes L^{-k})$ should consist of polarized sections of $(L^k\boxtimes L^{-k})\otimes\Omega_\P^{1/2} \cong L^{k-1}\boxtimes L^{-k+1}$. Polarized, in this case, means holomorphic, where the complex structure is reversed on the second factor of $S^2$.
For $k>0$, this is 
\[
A_k = \Hi_k \otimes\bar\Hi_k
\]
and $0$ for $k\leq0$.

In this case, $T^\t\Sigma = 0\boxplus TS^2$, so the restriction of the symplectic form on $\Sigma$ to this bundle is $-k\epsilon$. This gives the correction factor $\varepsilon_\P(x,x') = k\epsilon(x')/2\pi$.

Explicitly, for any $a,b\in A_k$ and $x,x'\in S^2$,
\[
(a*b)(x,x') = \int_{S^2} a(x,x'')b(x'',x') \sqrt{\tfrac{k}{2\pi}\epsilon(x'')}\  .
\]
This is holomorphic in $x$ because $a(x,x'')$ is, and it is antiholomorphic in $x'$ because $b(x'',x')$ is, therefore $a*b\in A_k$. 

The involution on $A_k$ is defined simply by $a^*(x,x') = \bar a(x',x)$.

It is simple to verify that this product and involution are the same as those in $\Li(\Hi_k)$ by looking at the multiplication of simple tensor products.

\begin{thm}
For any $a\in A_k$,
\[
I_k(a) = \frac{\unit^*(a)}{\sqrt{k\epsilon/2\pi}}\ .
\]
\end{thm}
\begin{proof}
First, consider the map $\Hi_k\to\co$ given by $\varphi\mapsto \langle\psi_o\vert\varphi\rangle$ and the map $\Hi_k\to L^k_o\otimes\sqrt{\F^\perp}\cong L^{k-1}_o$ given by evaluation $\varphi\mapsto\varphi(o)$. Equivariance under the $\T$-action implies that these must be proportional to each other, and so
\beq
\label{proportional}
\bar\varphi(o) \varphi(o) = C \langle\varphi\vert\psi_o\rangle\langle\psi_o\vert\varphi\rangle \sqrt{\epsilon(o)}
\eeq
for some constant $C\in\R$. 

Now consider the expression
\beq
\label{projector}
\lvert g\psi_o\rangle\langle g\psi_o\rvert \ .
\eeq
Because $U(1)$ leaves $\psi_o$ invariant up to phase, and phases cancel in this expression, this is really a function of $go\in S^2$. This means that we can meaningfully write
\beq
\label{average projector}
\int_{S^2} \lvert g\psi_o\rangle\langle g\psi_o\rvert \, \epsilon(go) \ .
\eeq
The expression \eqref{average projector} is, by construction, an $\SU(2)$-invariant operator on $\Hi_k$, which carries an irreducible representation, therefore \eqref{average projector} is a multiple of the identity. The expression \eqref{projector} has rank $1$ and $\int_{S^2}\epsilon = 2\pi$, therefore \eqref{average projector} has trace $2\pi$. Since $\dim\Hi_k=k$, this implies that 
\beq
\label{average formula}
\int_{S^2} \lvert g\psi_o\rangle\langle g\psi_o\rvert \, \epsilon(go) = \frac{2\pi}k\ .
\eeq

To determine the coefficient, $C$, consider an arbitrary $\varphi\in\Hi_k$ and insert the identity \eqref{average formula} into the computation of the norm:
\begin{align*}
\langle\varphi\vert\varphi\rangle &= \frac{k}{2\pi} \int_{S^2}  \langle\varphi\vert g\psi_o\rangle\langle g\psi_o\vert\varphi\rangle \, \epsilon(go) \\
&= \frac{k}{2\pi C} \int_{S^2} \bar\varphi(go) \varphi(go) \sqrt{\epsilon(go)} 
= \frac{\sqrt{k/2\pi}}{C} \langle\varphi\vert\varphi\rangle \ ,
\end{align*}
therefore $C=\sqrt{k/2\pi}$.

Equation \eqref{proportional} (with this normalization) can be rewritten as a statement about $a = \lvert\varphi\rangle\langle\varphi\rvert\in A_k$:
\begin{align*}
\unit^*(a)(o) = a(o,o) &= \varphi(o)\bar\varphi(o) \\
&=  \langle\psi_o\vert\varphi\rangle\langle\varphi\vert\psi_o\rangle \sqrt{\tfrac{k}{2\pi}\epsilon(o)} = I_k(a)(o) \sqrt{\tfrac{k}{2\pi}\epsilon(o)}  \ .
\end{align*}
Any $a\in A_k$ can be written as a linear combination of elements of this form, therefore this is true for any $a\in A_k$. By equivariance, this is true at any point of $S^2$, not just $o$.
\end{proof}

\subsubsection{Quantization of $\Hei(S^2,\epsilon)$}
\label{HS2}
Now, I shall examine the geometric quantization of $\Hei(S^2,\epsilon)$ using a symplectic groupoid. Although it is not really necessary for the quantization, it is illustrative to look at the structure of the groupoid $\Gamma$, which can be described quite explicitly here.

For convenience, let $G=\SU(2)$ and let $\T\subset G$ be the subgroup of diagonal matrices. The sphere $S^2$ is identified with the coset space $G/\T$, where $\T$ acts from the right. Let $X,Y,Z\in \su(2)$ be the standard generators with $[X,Y]=Z$, \emph{et cetera}.

The principal $\T$-bundle corresponding to the prequantization of $S^2$ is $G$ itself. The connection is defined by the left-invariant contact form, $\theta_0\in\Omega^1(G)$ such that, at the identity $e\in G$,  $\langle X,\theta_0\rangle = \langle Y,\theta_0\rangle = 0$ and $\langle Z,\theta_0\rangle = -\frac12$. 

The contact groupoid $\jac$ that integrates $(S^2,\epsilon)$ is the circle bundle associated to the prequantization of $\Sigma = \Pair S^2$. This is the quotient $\jac=\Pair(G)/\T$, where $\T$ acts by the right diagonal action on $\Pair G = G\times G$. The contact form $\theta\in\Omega^1(\jac)$ is such that its pullback to $\Pair G$ is $\t^*\theta_0-\s^*\theta_0$.

However, in order to exploit the symmetry here, it is useful to treat $\Pair(G)$ as the  action groupoid $G\ltimes G$, where $G$ acts on $G$ by right translations; that is, $g\in G$ transforms $g'\in G$ to $g' g^{-1}$, so the isomorphism $G\ltimes G\to \Pair(G)$ is given by $(g,g')\mapsto (g'g^{-1},g')$. 

This is just another way of identifying this groupoid with the manifold $G\times G$, but the advantage is that the contact form, $\theta$, is constant along the unit submanifold, $\{e\}\times G$. 
Along the unit submanifold, $\theta$ is normal to the unit submanifold and to the vectors $X$ and $Y$ on the first copy of $G$. 

Define $E \subset T_eG = \su(2)$ to be the subspace spanned by $X$ and $Y$. This makes $(G,\{e\},E)$ an explosive triple.

To construct $\Gamma$ by a double explosion, we need the subbundle of the normal bundle that is normal to $\theta$ (along the unit submanifold). That is the image of $E\times G$ under the quotient by $\T$.

With this identification, $\jac = (G\ltimes G)/\T$, where $u\in \T$ transforms $(g,g')\in G\ltimes G$ to $(u^{-1}gu,g'u) = (\Ad_{u^{-1}}g,g'u)$. 

For any $u\in\T$, the subspace $E\subset \su(2)$ is invariant under the adjoint action of $u^{-1}$, therefore $\Ad_{u^{-1}} : (G,\{e\},E)\to (G,\{e\},E)$ is a compatible map, and $u\in \T$ acts on the double explosion of $G$ by $\EE\Ad_{u^{-1}} : \EE(G,\{e\},E) \to \EE(G,\{e\},E)$. The symplectic groupoid integrating $\Hei(S^2,\epsilon)$ is the  double explosion,
\begin{align*}
\Gamma &= \EE(\jac,S^2,T^*S^2) = \EE(G\ltimes G,G,E\times G)/\T \\ &= (\EE(G,\{e\},E)\ltimes G)/\T \ .
\end{align*}

The groupoid $\EE(G,\{e\},E)$ is a family of simply connected Lie groups over $\R$. It is easiest to understand its structure by considering its Lie algebroid, which is a family of Lie algebras. The generators $X,Y,Z\in\su(2)$ are rescaled to $X'=\hbar X$, $Y'=\hbar Y$, and $Z'= \hbar^2 Z$. The Lie brackets of these are 
\begin{gather*}
[X',Y'] = \hbar^2 Z = Z'\ ,\\
[Y',Z'] = \hbar^3 X = \hbar^2 X'\ ,\\
\intertext{and}
[Z',X'] = \hbar^3 Y = \hbar^2 Y' \ .
\end{gather*}
From this, we can see that the group over any $\hbar\neq0$ is isomorphic to $G=\SU(2)$, but the group over $\hbar=0$ is the Heisenberg group.

The action of $\EE(G,\{e\},E)$ on $G$ is by right translations via the natural map $\Pr_G:\EE(G,\{e\},E) \to G$. In particular, at $\hbar=0$, this action is trivial.

So, as always, the subgroupoid of $\Gamma$ over any $\hbar\neq0$ is isomorphic to $\jac$. The subgroupoid $\Gamma\rvert_{\hbar=0}$ is a quotient of the product of the Heisenberg group (as a group) with $G$ (as a space). The result is a nontrivial bundle of groups over $S^2$, with fibers isomorphic to the Heisenberg group. The Heisenberg group is in particular a 3-dimensional vector space with a preferred direction. If we identify $S^2$ with the unit sphere in $\R^3$, then the copy of the Heisenberg group at a point of $S^2$ is the ambient  space and the preferred direction is the normal to $S^2$ there.

Because the polarization, $\P=\F\boxplus\bar\F$, of $\Sigma=\Pair S^2$ is totally complex, the polarization $\Q$ of $\Gamma$ becomes completely vertical at $\hbar=0$. The restriction $\iota^*\Q$ is real, and its leaves are the fibers of $\Gamma\rvert_{\hbar=0}$ over $S^2$. This means that the leaf space of $\iota^*\Q$ is $S^2$ and $\cs_{\iota^*\Q}(\Gamma\rvert_{\hbar=0},\iota^*\EA_\jac\theta) = \C(S^2)$.

Everything described in Section \ref{Multiply} applies here. The Bohr-Sommerfeld subvariety is the subgroupoid of $\Gamma$ where $\hbar=0$ or $\hbar^{-1}\in\Z$. In addition to the Bohr-Sommerfeld quantization conditions, the polarization $\Q$ is not positive where $\hbar<0$ and 
there do not exist any holomorphic sections over those components of $\Gamma_\BS$, so these should be discarded as well. With this in mind, define $\Gamma_{\BS,\geq0}$ as the subgroupoid of points where $\hbar=0$ or $\hbar^{-1}\in\N$.

Consider a polarized section $a\in \Gamma(\Gamma_{\BS,\geq0},\Omega_\Q^{1/2})$. The restriction of $a$ to $\hbar=k^{-1}$ is naturally identified with a polarized section $\ev_{k^{-1}}(a)$ of $(L^k\boxtimes L^{-k})\otimes \Omega_\P^{1/2}$ over $\Sigma$, i.e., $\ev_{k^{-1}}(a)\in A_k$. The definitions of the convolution products over $\Sigma$ and $\Gamma$ fit together, such that $\ev_{k^{-1}}$ is a homomorphism.

In this case, $m = \rk \P_0-\rk \D_0 = 1$. The restriction of $\omega_\Gamma$ to $T^\t\Gamma$ is $-\hbar^{-1}\epsilon$, so $\varepsilon = \epsilon/2\pi\hbar$, and
\begin{align*}
\ev_0(a) &= \lim_{\hbar\to0^+} \frac{\unit^*[\ev_\hbar (a)]}{\sqrt{\epsilon/2\pi\hbar}} 
= \lim_{k\to\infty} \frac{\unit^*[\ev_{k^{-1}}(a)]}{\sqrt{k\epsilon/2\pi}}
= \lim_{k\to\infty} I_k[\ev_{k^{-1}}(a)] \ .
\end{align*}
The quantization using $\Gamma$ gives a way of fitting the algebras $A_k$ and $\C(S^2)$ together; this shows that this agrees perfectly with the correspondence given by the coherent states maps $I_k$.

\subsection{Connes' tangent groupoid and Landsman's generalization}
The relationship between pseudodifferential operators on a manifold, $N$, and their symbols is an example of quantization. Connes \cite{Con94} has studied the analytic index (of a symbol) by essentially constructing a strict deformation quantization from $\C_0(T^*N)$ (the algebra of symbols) to $\K[L^2(N)]$ (the algebra of operators). To be precise, he constructs a continuous field of \cs-algebras over the interval $[0,1]$ such that the algebra over $0$ is $\C_0(T^*N)$ and over any other point is $\K[L^2(N)]$.

He constructs this geometrically from what he calls the ``tangent groupoid''  of $N$ (not to be confused with the tangent bundle to a groupoid, which is also a groupoid). Algebraically, this is a union of $TN\times\{0\}$ and $\Pair(N)\times(0,1]$. As a manifold, it is the part of the simple explosion $\EE(\Pair N,N)$ where $\hbar\in[0,1]$. (The restriction to $[0,1]$ is not important.)

In that case, the manifold being quantized is $T^*N$, which is of course the dual vector bundle to $TN$, which is the  Lie algebroid  integrated by $\Pair(N)$. This generalizes to any integrable Lie algebroid $A$. The (total space of the) dual vector bundle $A^*$ is a Poisson manifold. If $\G$ is a Lie groupoid integrating $A$, then $\cs(\G)$ is a quantization of $A^*$. 

Landsman \cite{lan2} has constructed a strict deformation quantization of $A^*$ by generalizing Connes construction to a continuous field of \cs-algebras over $\R$. The algebra over $0$ is $\C_0(A^*)$ and over any other point is $\cs(\G)$. He generalizes Connes' tangent groupoid to what he calls the ``normal groupoid''. Algebraically, this is a union of $A\times\{0\}$ (treated as a vector bundle and hence a groupoid) with $\G\times\R_*$. As a manifold, it is the simple explosion, $\EE(\G,M)$, of $\G$ along the unit submanifold.

(Van Erp's parabolic tangent groupoid is different generalization of Connes' tangent groupoid and is discussed in Appendix~\ref{Parabolic}.)

\subsubsection{Quantization of $A^*$}
Landsman's construction can be recovered from what I have presented here.  Let $A$ be a Lie algebroid over $N$ and $\G$ a groupoid integrating $A$. The cotangent bundle $T^*\G$, with the standard symplectic form, is a symplectic groupoid integrating the Poisson manifold $A^*$.

Let $\theta_0\in\Omega^1(T^*\G)$ be the Liouville form. This is a multiplicative  symplectic potential, so the symplectic groupoid $(T^*\G,-d\theta_0)$ is prequantized by the trivial line bundle, $\co\times T^*\G$, with connection $d+i\theta_0$ and the trivial cocycle. 

Let $p : T^*\G\to \G$ be the bundle projection. This is a fibration of groupoids, so  $\P:=\ker Tp$ is a groupoid polarization; it is the foliation of $T^*\G$ whose leaves are the fibers of the cotangent bundle. Since the leaves are Lagrangian, $\P$ is a symplectic groupoid polarization.
The Liouville form is adapted (normal) to the polarization, $\P$, so the twisted, polarized convolution algebra of $T^*\G$ is just the convolution algebra of $\G$:
\[
\cs_\P(T^*\G,\theta_0) = \cs(\G)\ .
\]

\subsubsection{Quantization of $\Hei A^*$}
Let $\eta\in\X^1(A^*)$ be the Euler vector field.
 The Poisson structure $\Pi\in\X^2(A^*)$ is homogeneous in the sense that, $\Lie_\eta\Pi=-\Pi$. This means that $\Pi$ is itself the Poisson coboundary, $\Pi = -[\eta,\Pi]$, of $\eta$.
 
In general, the Jacobi algebroid of a Poisson manifold is an extension of the cotangent Lie algebroid determined by $\Pi$ as a Poisson cocycle, but because $\Pi$ is exact in this case, the Jacobi algebroid is isomorphic to the trivial extension; an isomorphism $r:T^*_\Pi A^*\to T^* A^*\oplus\R$ is defined by $r(\xi,f) = (\xi,f+\eta\inner\xi)$ for $\xi\in T^*A^*$ and $f\in\R$.
 This implies that the Jacobi algebroid is integrated by the Cartesian product $\jac := T^*\G\times\R$ with the group $\R$. If $z$ denotes the coordinate on the group $\R$, then the contact form on $\jac$ is $\theta=\theta_0+dz$. (This is a trivial special case of the structure given by Theorem~\ref{R fibers}.)

The projection $\pro:\jac\to T^*\G$ is just the Cartesian projection. The composition $p\circ\pro:\jac\to \G$ is also a fibration of groupoids. It maps the unit submanifold, $A^*$, to the unit submanifold, $N$, therefore it is a compatible map from the explosive triple $(\jac,A^*,T^*A^*)$ to $(\G,N,A)$. Exploding this gives another fibration of groupoids,
\[
\EE(p\circ\pro) : \Gamma \to E(\G,N) \ .
\]

Over $\hbar\neq0$, the kernel foliation is $\ker T(p\circ\pro)\times 0 = (T\pro)^{-1}\P\times 0$. This shows that $\Q=\ker T\EE(p\circ\pro)$ is the polarization of $\G$ constructed from $\P$.

The symplectic potential $\EA_\jac \theta$ is multiplicative but is not adapted to the polarization $\Q$. However, as in eq.~\eqref{corrected connection} above we can add an exact form to get the symplectic potential $\EA_\jac \theta-d(\hbar z) = z\,d\hbar + \EA_\jac \theta_0$, which is adapted. The function $\hbar z$ is multiplicative, therefore this adapted symplectic potential is multiplicative. This implies that the twisted, polarized convolution algebra of $\Gamma = \EE(\jac,A^*,T^*A^*)$ is just the convolution algebra of Landsman's normal groupoid, $\EE(\G,N)$:
\[
\cs_\Q(\Gamma,z\,d\hbar + \EA_\jac \theta_0) = \cs[\EE(\G,N)]\ .
\]

In this way, my construction with the  polarization $\P = \ker Tp$ reproduces Landsman's construction, which for $\G=\Pair(N)$ reproduces Connes' construction.

\section{Other Deformations}
\label{Other}
The Heisenberg-Poisson manifold $\Hei(M,\Pi)$  is $M\times \R$ with the Poisson structure $\hbar \Pi$. We might more generally consider some $\hbar$-dependent Poisson structure, $\pi(\hbar)$, that does not necessarily vary linearly. This is equivalent to considering a Poisson structure on $M\times\R$ for which $\hbar$ is a Casimir function.

This may plausibly be quantized to construct a continuous field of \cs-algebras. The algebra at $\hbar$ would be the quantization of $M$ with the Poisson structure $\pi(\hbar)$ and the algebra of continuous sections would be the quantization of $(M\times\R,\pi)$. 

To construct a strict deformation quantization of $M$, we want the algebra at $\hbar=0$ to be $\C_0(M)$, therefore we should demand that $\pi(0)=0$. What is the role of the original Poisson structure, $\Pi$, on $M$?

Suppose that quantizing $(M,\pi(\hbar))$  gives a \cs-algebra $A_\hbar$ and a linear map $Q_\hbar : \C^\infty_0(M)\to A_\hbar$ that satisfies the Dirac quantization rule,
\[
[Q_\hbar(f),Q_\hbar(g)] \approx i Q_\hbar(\{f,g\}_{\pi(\hbar)}) 
\]
for all $f,g\in \C^\infty_0(M)$. (I am being deliberately vague about the approximation, but it should converge rapidly in the limit of vanishing Poisson structure.)

If we want to construct a strict deformation quantization of $(M,\Pi)$, then we need the Dirac condition,
\[
[Q_\hbar(f),Q_\hbar(g)] \approx i\hbar(\{f,g\}_\Pi)\ .
\]
Comparing these shows that we need $\pi(\hbar)\approx \hbar\Pi$ for $\hbar\approx0$. In other words $\pi(0)=0$ and $\pi'(0)=\Pi$; the given Poisson structure should be the derivative of the Poisson structure used in constructing a  quantization.

\subsection{Integrability}
Integrability is a very restrictive condition on Poisson manifolds of this type. Consider the case that both $\Pi$ and  $\pi(\hbar)$ for $\hbar\neq0$ are invertible (hence symplectic). 

\begin{thm}
\label{Nonintegrable}
Let $(M,\Pi)$ be a simply connected Poisson manifold, and consider a Poisson structure on $M\times\R$ given by a variable Poisson structure $\pi(\hbar)\in\X^2(M)$ such that $\pi(0)=0$ and $\pi'(0)=\Pi$. Moreover, suppose $\Pi$ and $\pi(\hbar)$ are the inverses of symplectic forms $\Omega$ and $\omega(\hbar)$ for $\hbar\neq0$. If this Poisson structure on $M\times\R$ is integrable, then the cohomology class of $\omega(\hbar)$ takes the form
\[
[\omega(\hbar)] = \rho + \frac{[\Omega]}{F(\hbar)}\ ,
\]
where $\rho\in H^2(M;\R)$ and $F(\hbar)$ is a smooth function with $F(0)=0$, $F'(0)=1$, and $F'(\hbar)>0$ for all $\hbar\in\R$.
\end{thm}
\begin{proof}
Crainic and Fernandes \cite{c-f2} showed that a Poisson manifold is integrable if and only if its monodromy groups are locally uniformly discrete.

The monodromy groups of a Poisson manifold are subgroups of the conormal spaces to the symplectic foliation. More specifically, the conormal space at any point is a Lie algebra and the monodromy group lies in its center.

In this case, the conormal space is the whole cotangent fiber at points where $\hbar=0$ and is spanned by $d\hbar$ at points where $\hbar=0$. However, at points where $\hbar=0$, the conormal fiber is isomorphic to a Heisenberg Lie algebra, with center spanned by $d\hbar$, therefore at \emph{any} point of $M\times\R$, the monodromy group lies in the 1-dimensional space spanned by $d\hbar$ there.

At the regular points (those where $\hbar\neq0$) the monodromies are easy to compute \cite{c-f2} as variations of symplectic area. 
If $j: S^2\to M$, then the symplectic area
\[
\int_{S^2} j^*[\omega(\hbar)]
\]
only depends upon the homotopy class of $j$, and the monodromy of $j$ is
\[
\frac{d}{d\hbar} \int_{S^2} j^*[\omega(\hbar)]
\]
times $d\hbar$. If $(M\times \R,\pi)$ is integrable, then the monodromy groups must be discrete, therefore the ratios between these monodromies must be rational. The ratios are also continuous functions of $\hbar$, therefore they are constant.

Because $M$ is  simply connected,  the Hurewicz homomorphism, 
\[
\pi_2(M)\to H_2(M;\Z) \ ,
\] 
is surjective, so the monodromies completely determine the cohomology class  $[\omega'(\hbar)]$, therefore this is the product of some function of $\hbar$ with a constant cohomology class. 

For $\hbar\approx 0$, $\omega(\hbar)\approx \Omega\, \hbar^{-1}$, thus $\omega'(\hbar)\approx - \Omega\, \hbar^{-2}$. The cohomology class  $[\omega'(\hbar)]$ is thus proportional to $[\Omega]$. Integrating shows that 
\[
[\omega(\hbar)] = \rho + \frac{[\Omega]}{F(\hbar)} \ ,
\]
where  $\rho\in H^2(M)$, and $F$ is a real function of $\hbar$ with $F'(0)=1$. 

Moreover, if  $F'$ vanished at some value of $\hbar$, then the monodromies would collapse to $0$ and not be locally \emph{uniformly} discrete, therefore $F'>0$ and $F$ is a monotone increasing function.
\end{proof}

More generally, integrability becomes less restrictive the more degenerate the Poisson structure $\pi(\hbar)$ is. That gives more room for the monodromy groups to be discrete.

\subsection{Irrational symplectic manifolds}
\label{Irrational}
Suppose that $(M,\Omega)$ is a simply connected symplectic manifold with nonintegral symplectic form in the sense that $[\tfrac\Omega{2\pi}]\in H^2(M;\R)$ is not integral. This manifold is not prequantizable, so geometric quantization does not provide a corresponding noncommutative algebra. Indeed, there does not seem to be any algebra that plays the role of the natural quantization of $(M,\Omega)$.

This does not necessarily mean that there does not exist a strict deformation quantization. For instance, if some nonzero real multiple of $\Omega$ is integral, then there is really no problem. We can still quantize $\Hei(M,\Pi)$, because there are infinitely many values of $\hbar$ where the symplectic submanifold of $\Hei(M,\Pi)$ is prequantizable.

On the other hand, if $\Omega$ is \emph{not} a multiple of an integral $2$-form, then we need to take a more general approach, because there is no good quantization of $\Hei(M,\Pi)$. In this case, we should try varying the Poisson structure nonlinearly, as in the previous section.
\begin{remark}
I will now be using Cartesian product notation to present a Poisson or symplectic structure on $S^2\times S^2$.
\end{remark}

The simplest example in which this works is symplectic. As in Section \ref{Sphere}, let $\epsilon\in\Omega^2(S^2)$ be half the volume form of the unit sphere. Let $\epsilon^{-1}$ denote the corresponding Poisson bivector. 

Consider $(S^2\times S^2,\Omega)$ with the irrational symplectic form $\Omega = \epsilon\times\frac{-1+\sqrt5}2\epsilon$, or equivalently, the Poisson structure $\Pi = \epsilon^{-1}\times\frac{1+\sqrt5}{2}\epsilon^{-1}$. The Heisenberg-Poisson manifold is unquantizable, because as the symplectic structure is rescaled the two components are never simultaneously integral.  Instead, we can consider $S^2\times S^2\times\R$ with the varying Poisson structure
\beq
\label{Fibonnacci Poisson}
\pi(\hbar) = \hbar\epsilon^{-1}\times\tfrac{2\hbar}{-1+\sqrt{5+4\hbar^2}}\epsilon^{-1} 
\eeq
whose inverse is 
\[
\omega(\hbar) =\hbar^{-1}\epsilon\times\tfrac{-1+\sqrt{5+4\hbar^2}}{2\hbar}\epsilon \ .
\]
With this choice, $\pi(0)$ and $\pi'(0)$ have the desired values. 
This is a strange looking choice, but unlike the Heisenberg-Poisson manifold, it is quantizable. The motivation for this  choice is discussed in \cite{h-h}.

Let $F_n$ denote the $n$'th Fibonacci number: $F_0=0$, $F_1=1$, $F_2=1$, \emph{et cetera}. 
Setting $\hbar=\pm F_{2k}^{-1}$ for $k=1,2,3\dots$, gives,
\[
\omega(\pm F_{2k}^{-1}) = (\pm F_{2k}\epsilon)\times(\pm F_{2k-1}\epsilon)
\]
which is integral! These are the only values of $\hbar$ for which this symplectic form is integral.

It is fairly clear how to quantize $(S^2\times S^2\times \R,\pi)$. The \cs-algebra is the algebra of  sections of a continuous field of \cs-algebras over 
\[
I = \{0,F_{2k}^{-1} \mid k\in\N\}\subset\R\ .
\]
The algebra over $\hbar=0$ is $\C(S^2\times S^2)$. The algebra over $\hbar=F_{2k}^{-1}$ is a matrix algebra, equal to the tensor product of the quantizations of $(S^2,F_{2k}\epsilon)$ and $(S^2,F_{2k-1}\epsilon)$; this is the algebra of square matrices of size $F_{2k}F_{2k-1}$.

Although this Poisson manifold has an obvious quantization, by Theorem~\ref{Nonintegrable} it is not integrable. The quantization cannot be constructed from a symplectic groupoid, because there is no symplectic groupoid. To construct this quantization geometrically, we need to generalize the construction somehow.

There are many nonintegrable Poisson manifolds that should be quantizable. Stacky symplectic groupoids \cite{zhu1} are a general structure for dealing with nonintegrable Poisson manifolds, but  for this example a more elementary approach is sufficient.

Let $(\Gamma,\omega_\Gamma)$ be the symplectic groupoid integrating $\Hei(S^2,\epsilon)=(S^2\times\R,\hbar\,\epsilon^{-1})$ that was constructed in Section~\ref{HS2}, and consider the product 
\[
(\Gamma\times\Gamma,\omega_\Gamma\times\omega_\Gamma)\ .
\]
This is a symplectic groupoid over  $\Hei S^2\times\Hei S^2 = S^2\times\R\times S^2\times\R$. There are now two ``$\hbar$''s, so I will denote that coordinate on the second factor as $\lambda$. Instead, let $\G\subset\Gamma\times\Gamma$ be the submanifold  where 
\[
\lambda = \frac{2\hbar}{-1+\sqrt{5+4\hbar^2}} \ ;
\]
see Figure~\ref{Curve}.
This is a coisotropic subgroupoid over $S^2\times S^2\times\R$. The pullback of the symplectic form to $\G$ is automatically closed and multiplicative, but it is degenerate, so it makes $\G$  a \emph{presymplectic} groupoid.
\begin{figure}
\begin{center}
\includegraphics{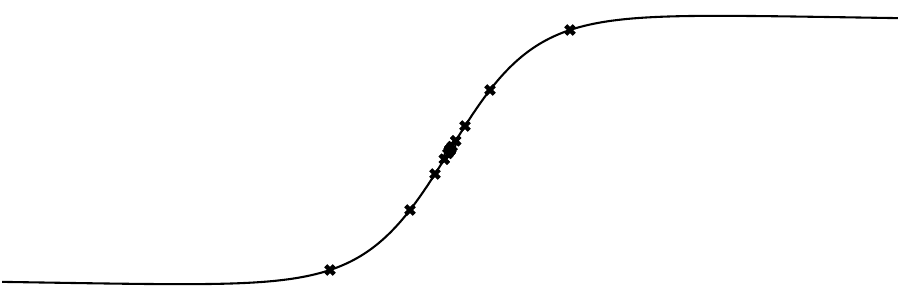}
\end{center}
\caption{The image of $\G$ in the $\hbar$-$\lambda$-plane. The marked points are where the Bohr-Sommerfeld conditions are satisfied.\label{Curve}}
\end{figure}

In principle, a symplectic groupoid over $S^2\times S^2\times\R$ should be obtained by symplectic reduction of $\G$, but no such symplectic groupoid exists, so what goes wrong? 

As we have seen, except at $\hbar=0$, the orbits of $\#d\hbar$ in $\Gamma$ are circles. In $\Gamma\times\Gamma$, these become tori. The degenerate direction of the presymplectic form on $\G$ is tangent to these tori. At most points, this gives a Kronecker foliation. The symplectic reduction of $\G$ would be the leaf space of the (1-dimensional) foliation by the degenerate direction of the presymplectic form, but that is not a smooth manifold in this case.

As in Section~\ref{HS2}, let $\Q$ be the polarization of $\Gamma$ derived from the complex structure on $S^2$. This gives an obvious polarization $\Q\times \Q$ on the larger symplectic groupoid. This is tangent to the submanifolds of constant $\hbar$ and $\lambda$, therefore it is tangent to $\G$. Define $\Rp$ as the restriction of $\Q\times\Q$ to $\G$.

This inherits most of the properties of $\Q\times\Q$. It is involutive, multiplicative, and Hermitian, therefore it is a groupoid  polarization of $\G$. Moreover, $\Rp$ includes the degenerate direction of the presymplectic form.

This $\Rp$ is the restriction of a Lagrangian distribution, so it is coisotropic. The rank of $\Rp$ is $6$ and the dimension of $\G$ is $11$, so $\Rp$ is maximally coisotropic. \emph{Maximal coisotropic} seems like the most natural generalization of \emph{Lagrangian} to presymplectic manifolds, so it is reasonable to consider $\Rp$ to be a polarization of $\G$ as a presymplectic groupoid.

The prequantization of $\G$ is just the restriction of the prequantization of $\Gamma \times \Gamma$. Except at $\hbar=0$, the real part of the polarization $\Rp$ is a foliation of $\G$ by 2-tori. This means that there are 2 independent Bohr-Sommerfeld conditions coming from holonomies around these tori. 

These are just the Bohr-Sommerfeld conditions for each copy of $\Gamma$, i.e., $\hbar^{-1},\lambda^{-1}\in\Z$. Along $\G$, this means that either $\hbar=\lambda=0$ or $\hbar=\pm F_{2k}^{-1}$ and $\lambda=\pm F_{2k-1}^{-1}$ for some $k\in\N$; see Figure~\ref{Curve}.

The polarization, $\Rp$, fails to be positive when $\hbar,\lambda<0$. Applying the Bohr-Sommerfeld conditions and requiring positivity leaves $\G_{\BS,\geq0}$, which is the subgroupoid of $\Gamma\times\Gamma$ where either $\hbar=\lambda=0$ or $\hbar= F_{2k}^{-1}$ and $\lambda= F_{2k-1}^{-1}$ for some $k\in\N$. 

Let $\G_k$ be the $k$'th component of $\G_{\BS,\geq0}$. This is a $\T^2$-bundle over $\Pair(S^2)^2$. Polarized functions over $\G_k$ are equivalent to polarized sections of a nontrivial line bundle over $\Pair(S^2)^2$. In this way, geometric quantization produces the correct matrix algebra from $\G_k$.

The result is that geometric quantization with this presymplectic groupoid produces the expected quantization of $(S^2\times S^2\times\R,\pi)$, even though no symplectic groupoid exists in this situation.

This approach may allow the geometric quantization of many other nonintegrable Poisson manifolds. The Lie algebroid of this contact groupoid is a central extension of the cotangent Lie algebroid. Centrally extending enlarges the vector spaces in which the monodromy groups live. This makes it easier for the monodromy groups to be discrete, and thus for the Lie algebroid to be integrable.

\begin{appendix}
\section{Exploding Contact Groupoids}
\label{Exploding contact}
For the constructions in this paper, I have only used \emph{strict} contact groupoids, for which the contact form is multiplicative. On a general contact groupoid (Def.~\ref{Contact_groupoid}) the contact form is only multiplicative in a twisted sense, and the base manifold is not a Poisson manifold but a Jacobi manifold. I shall briefly show how to generalize Theorem \ref{Integration thm}, because it provides a class of examples of symplectic groupoids.

In order to do this, I will need one more general result about explosions, showing how the $\hbar$ coordinate on an explosion of $M$ can be rescaled by a function on $M$.
\begin{lem}
\label{Rescale}
Given an explosive triple $(M,N,E)$, for any smooth real function $r\in\Ci(M)$, there exists a unique smooth map $\Res_r:\EE(M,N,E)\to \EE(M,N,E)$ such that
\begin{enumerate}
\item \label{Rescale1}
$\Pr_M\circ \Res_r= \Pr_M$.
\item \label{Rescale2}
$\Res_r^*(\hbar) =\Pr_M^*(e^{-r})\cdot \hbar\in\Ci[\EE(M,N,E)]$.
\item\label{Rescale3}
For a compatible map $\phi : (M,N,E)\to (M',N',E')$ and a smooth real function $r\in\Ci(M')$,
\begin{equation}
\label{rescaling}
\begin{tikzcd}
\EE(M,N,E) \arrow{d}{\Res_{\phi^*r}} \arrow{r}{\EE\phi} & \EE(M',N',E') \arrow{d}{\Res_r} \\
\EE(M,N,E) \arrow{r}{\EE\phi} & \EE(M',N',E')
\end{tikzcd}
\end{equation}
is a commutative diagram.
\item \label{Rescale4}
For any two smooth real functions $r$ and $s$, $\Res_r\circ\Res_s = \Res_{r+s}$.
\end{enumerate}
\end{lem}
\begin{proof}
Once again, I first consider the case that $(M,N,E)$ is an explosive chart. In terms of the coordinates on $\EE(M,N,E)$, 
\[
\Pr_M (x,y,z,\hbar) = (x,\hbar y,\hbar^2 z)
\]
and $\hbar$ is just the coordinate of that name. Conditions \ref{Rescale1} and \ref{Rescale2} are satisfied if and only if
\[
\Res_r(x,y,z,\hbar) = (x,e^{r(x,\hbar y,\hbar^2 z)}y,e^{2r(x,\hbar y,\hbar^2 z)}z,e^{-r(x,\hbar y,\hbar^2 z)}\hbar)\ .
\]
This is explicitly a smooth map.

Condition \ref{Rescale3} follows by a direct computation from this explicit formula. This works because conditions \ref{Rescale1} and \ref{Rescale2} determine $\Res_r$ uniquely.

The unique extension of this construction to any explosive triple can be constructed by patching together explosive charts. Condition \ref{Rescale3} implies that this is consistent.

Finally, to check condition \ref{Rescale4}, it is sufficient to compute the compositions of $\Res_r\circ\Res_s$ and $\Res_{r+s}$ with $\Pr_M$ and $\hbar$.
\end{proof}

\begin{lem}
\label{Form rescaling}
For any $\theta\in\Omegan(M,N,E)$,
\[
\Res_r^*(\EA_M\theta) =  \EA_M(e^r\theta)\ .
\]
\end{lem}
\begin{proof}
The defining property of $\EA_M \theta$ is that $\Pr_M^*\theta = \hbar \EA_M \theta$, so
\begin{align*}
\hbar\Res_r^*(\EA_M\theta) &= \Pr_M^*(e^r)\Res_r^*(\hbar\EA_M\theta) 
= \Pr_M^*(e^r)\Res_r^*\Pr_M^*\theta\\ &= \Pr_M^*(e^r\theta) = \hbar\EA_M(e^r\theta)\ . \qedhere
\end{align*}
\end{proof}

\begin{definition}
A \emph{Jacobi manifold} $(M,\Lambda,R)$ is a manifold $M$ with a bivector field $\Lambda \in \X^2(M)$ and vector field $R\in\X^1(M)$  such that $[\Lambda,\Lambda]=2R\wedge \Lambda$ and $[\Lambda,R]=0$. A \emph{Jacobi map} is a smooth map $\phi:M\to M'$ that pushes forward the bivector and vector on $M$ to the bivector and vector on $M'$. 
\end{definition}
In particular, a contact form, $\theta\in\Omega^1(M)$, determines a Jacobi structure by the conditions that $R\inner \theta=1$, $R\inner d\theta = 0$ and for any $\xi\in T^*M$,
\[
\xi = -(\#_\Lambda\xi)\inner d\theta + (R\inner\xi)\theta \ .
\]

I already remarked in the proof of Theorem~\ref{Integration thm} on the relationship between Weinstein's Heisenberg-Poisson construction \cite{wei4}, $\Hei$, and Lichnerowicz's Poissonization construction for Jacobi manifolds \cite{lic,c-z}. For a Poisson manifold, $(M,\Pi)$, the Poissonization is isomorphic to the  $\hbar>0$ part $\Hei_+(M,\Pi)\subset \Hei(M,\Pi)$, by the coordinate transformation, $\hbar=e^{-s}$. This suggests extending the definition of $\Hei$ to the category of Jacobi manifolds in a way that  maintains this relationship.
\begin{definition}
For any Jacobi manifold, $(M,\Lambda,R)$, let $\Hei(M,\Lambda,R)$ be the manifold $M\times\R$ with the Poisson bivector
\[
\hbar\Lambda - \hbar^2\frac{\partial}{\partial \hbar}\wedge R \ ,
\]
and for any Jacobi map, $\phi$, let $\Hei \phi:=\phi\times\id_\R$.
\end{definition}
This $\Hei$ is a functor from the category of Jacobi manifolds to the category of Poisson manifolds

\begin{definition}
\label{Contact_groupoid}
A \emph{contact groupoid} $(\jac,\theta)$ is a Lie groupoid, $\jac$, with a contact form $\theta\in\Omega^1(\jac)$ such that there exists a real function (the \emph{Reeb function}) $r \in\Ci(\jac)$ for which
\beq
\label{twisted mult}
0 =  \pr_1^*\theta - \m^*\theta + \pr_1^*(e^r)\pr_2^*\theta\ .
\eeq
A contact groupoid, $(\jac,\theta)$, \emph{integrates} a Jacobi manifold, $(M,\Lambda,R)$ if $\jac$ is a groupoid over $M$ and $\t:\jac\to M$ is a Jacobi map.
\end{definition}
\begin{remark}
This definition differs slightly from that in \cite{k-s} and \cite{c-z}, because those authors chose to make the \emph{source} map Jacobi. My definition is consistent with my convention for symplectic groupoids. This choice is more significant here, because for a symplectic groupoid the source map is anti-Poisson, but for a contact groupoid it has no nice property.
\end{remark}
Because $\theta$ is nonvanishing, it uniquely determines the function $r$ (when it exists). This definition implies that the Reeb function is a groupoid $1$-cocycle --- or equivalently, a homomorphism $r:\jac\to\R$. If $r=1$, then this is a strict contact groupoid as in Definition~\ref{Contact}.

For any contact groupoid, $(\jac,\theta)$, the cocycle property of the Reeb function, $r$, implies that $\unit^*r=0$. Pulling eq.~\eqref{twisted mult} back by $(\unit,\unit):M\to\jac_2$ shows that $\unit^*\theta=0$, so $\theta$ is normal to the unit submanifold. Because $\theta$ is nonvanishing, it determines a corank 1 subbundle, $\theta^\perp \subset \nu_M\jac$ of the normal bundle. This means that we have an explosive triple $(\jac,M,\theta^\perp)$, such that by definition $\theta\in\Omegan(\jac,M,\theta^\perp)$.

\begin{definition}
\label{Poissonization2}
Given a contact groupoid $(\jac,\theta)$ over $M$ with Reeb function, $r$, define $\Gamma := \EE(\jac,M,\theta^\perp)$ and $\Gamma_2 := \EE(\jac_2,M,\theta^\perp\oplus \theta^\perp)$ with the structure maps 
\begin{gather*}
\unit_\Gamma := \EE\unit\ ,\qquad \t_\Gamma := \EE\t\ ,\qquad \s_\Gamma := \EE\s\circ \Res_{r}\ ,\qquad \inv_\Gamma := \EE\inv\circ\Res_{r}\ ,\\
\pr_{1\,\Gamma} := \EE\pr_1\ , \qquad \m_\Gamma:= \EE\m\ ,\qquad  \pr_{2\,\Gamma} := \EE\pr_2\circ\Res_{\pr_1^* r}\ ,
\end{gather*} 
and $2$-form $\omega_\Gamma := -d(\EA_\jac\theta)$.
\end{definition}
\begin{remark}
This definition is chosen precisely so that the restriction of $\Gamma$ to $\hbar>0$ is isomorphic to $\jac\times_r\R$, as defined in \cite[Def.~2.3]{c-z}.
\end{remark}
\begin{thm}
\label{Contact explosion}
If $(\jac,\theta)$ is a contact groupoid integrating $(M,\Lambda,R)$, then 
$(\Gamma,\omega_\Gamma)$ is a symplectic groupoid integrating $\Hei(M,\Lambda,R)$.
\end{thm}
\begin{proof}
The proof that this is a groupoid is a generalization of the proof of Lemma~\ref{Groupoid}. It was already proven there that $\t_\Gamma=\EE\t$ is a surjective submersion and its fibers are Hausdorff. The fibers of $\t_\Gamma$ are connected, because those of $\t$ are.

The other axioms of a Lie groupoid in Definition~\ref{Groupoid def} are expressed through commutative diagrams. For each of the diagrams below, the individual commutative squares and triangles are given by:
\begin{itemize}
\item
applying the functor $\EE$ to a commutative diagram for $\jac$,
\item
an instance of Lemma~\ref{Rescale}(\ref{Rescale3}),
\item
an instance of Lemma~\ref{Rescale}(\ref{Rescale4}), or
\item
the definition of one of the structure maps in Definition \ref{Poissonization2}.
\end{itemize}  

First, consider the diagram,
\[
\begin{tikzcd}
\Gamma_2 \arrow{r}{\Res_{\pr_1^*r}} \arrow{d}[swap]{\pr_{1\,\Gamma}=\EE\pr_1} \arrow[bend left=50]{rr}{\pr_{2\,\Gamma}} & \Gamma_2 \arrow{r}{\EE\pr_2} \arrow{d}{\EE\pr_1}& \Gamma \arrow{d}{\EE\t=\t_\Gamma}\\
\Gamma \arrow{r}{\Res_r} \arrow[bend right]{rr}[swap]{\s_\Gamma} & \Gamma \arrow{r}{\EE\s} & \Hei M \period
\end{tikzcd}
\]
The left square is commutative by \ref{Rescale}(\ref{Rescale3}). The right square is commutative because it is given by applying $\EE$ to $\t\circ \pr_2 = \s\circ\pr_1$. The triangles are definitions of structure maps. This shows that $\t_\Gamma\circ \pr_{2\,\Gamma} = \s_\Gamma\circ\pr_{1\,\Gamma}$, and so 
$(\pr_{1\,\Gamma},\pr_{2\,\Gamma}):\Gamma_2\to\Gamma \mathbin{{}_{\s_\Gamma}\times_{\t_\Gamma}} \Gamma$.
The proof that this is a diffeomorphism is essentially the same as in  Lemma~\ref{Groupoid}. This shows that $\Gamma_2$ and the projection maps $\pr_{1\,\Gamma}$ and $\pr_{2\,\Gamma}$ are correctly defined.

Similarly, $\Gamma_3=\EE(\jac_3,M,\theta^\perp\oplus\theta^\perp\oplus\theta^\perp)$. The various face maps can be constructed by exploding those for $\jac_3$, and modifying some with $\Res$. Associativity \eqref{associativity} can be proven in this way, but the explicit proof  requires too much notation to present it here.

 Next, to prove (\ref{s and t}a), consider the diagram,
\[
\begin{tikzcd}
\Gamma_2 \arrow[bend left=50]{rr}{\pr_{2\,\Gamma}} \arrow{r}{\Res_{\pr_1^*r}}\arrow{dr}[swap,inner sep=-.3ex]{\Res_{\m^*r}} \arrow[swap]{dd}{\m_\Gamma=\EE\m}
& \Gamma_2 \arrow{r}{\EE\pr_2} \arrow[rightarrow]{d} \arrow{d}{\Res_{\pr_2^*r}} 
& \Gamma \arrow[swap]{d}{\Res_r} \arrow[bend left]{dd}{\s_\Gamma}\\ 
& \Gamma_2 \arrow{d}{\EE\m}\arrow{r}{\EE\pr_2} 
& \Gamma\arrow[swap]{d}{\EE\s}\\
\Gamma \arrow{r}{\Res_r} \arrow[bend right,swap]{rr}{\s_\Gamma}
& \Gamma \arrow{r}{\EE\s} 
& \Hei M \period
\end{tikzcd}
\]
The upper right triangle is commutative by the cocycle property, $\pr_1^*r+\pr_2^*r=\m^*r$, and Lemma~\ref{Rescale}(\ref{Rescale4}). The lower right square is $\EE$ applied to (\ref{s and t}a) for $\jac$. The other two corners are commutative by Lemma~\ref{Rescale}(\ref{Rescale3}). This shows that the diagram is commutative, which proves (\ref{s and t}a) for $\Gamma$. More simply, (\ref{s and t}b) for $\Gamma$ is given by applying $\EE$ to (\ref{s and t}b) for $\jac$.

Next consider the diagrams \eqref{right unit}, which express the right unit axiom. For $\jac$, the map that completes these diagrams is $F := (\id_\jac,\unit\circ\s)$. The map $\EE F$ plays this role for $\Gamma$. The first two diagrams for $\Gamma$ are given by applying $\EE$ to the diagrams for $\jac$. The third diagram for $\Gamma$ comes from 
\[
\begin{tikzcd}
\Gamma \arrow{r}{\EE F} \arrow{d}{\Res_r} \arrow[bend right]{dd}[swap]{\s_\Gamma}  & \Gamma_2 \arrow{d}{\Res_{\pr_1^*r}} \arrow[bend left=70]{dd}{\pr_{2\,\Gamma}}\\
\Gamma \arrow{r}{\EE F}\arrow{d}{\EE\s} &  \Gamma_2 \arrow{d}{\EE\pr_2} \\
\Hei M \arrow{r}{\EE\unit}[swap]{ = \unit_\Gamma} &  \Gamma \period
\end{tikzcd}
\]
The bottom square is given by applying $\EE$, the top square is commutative because $F^*(\pr_1^*r) = (\pr_1\circ F)^*(r)=r$.
The left unit axiom \eqref{left unit} can be proven similarly using the map $(\unit\circ\t,\id_\jac):\jac\to\jac_2$.

The cocycle property of the Reeb function implies that $\inv^*r=-r$. With (\ref{s inv t}a) for $\jac$, this implies that 
\[
\begin{tikzcd}
\Gamma \arrow[bend left=40]{rr}{\inv_\Gamma}\arrow{r}{\Res_r} \arrow{dr}[swap]{\id_\Gamma} 
& \Gamma \arrow{r}{\EE\inv}\arrow{d}{\Res_{-r}}  &\Gamma \arrow{d}[swap]{\Res_r} \arrow[bend left]{dd}{\s_\Gamma} \\
& \Gamma \arrow{r}{\EE\inv}\arrow{dr}[swap,inner sep=0.25ex]{\t_\Gamma=\EE\t} & \Gamma\arrow{d}[swap]{\EE\s} \\
& & \Hei M 
\end{tikzcd}
\]
is commutative, which gives (\ref{s inv t}a) for $\Gamma$. More trivially, applying $\EE$ to (\ref{s inv t}b) for $\jac$ gives that diagram for $\Gamma$.

The diagram in the left inverse axiom \eqref{left inverse} for $\jac$ are completed with the map 
$G = (\inv,\id_\jac)$.  The diagrams \eqref{left inverse} for $\Gamma$ come from
\[
\begin{tikzcd}
\Gamma 
\arrow{r}{\Res_r} \arrow[bend right=20]{rrd}[swap,inner sep=0ex]{\inv_\Gamma} & \Gamma \arrow{r}{\EE G} \arrow{dr}[swap,inner sep=0ex]{\EE \inv}& \Gamma_2 \arrow{d}{\EE\pr_1=\pr_{1\,\Gamma}} \\
& & \Gamma
\end{tikzcd}
\qquad
\begin{tikzcd}
\Gamma 
\arrow{r}{\Res_r} \arrow{dr}[swap]{\id_\Gamma}  & \Gamma \arrow{r}{\EE G}\arrow{d}{\Res_{-r}} & \Gamma_2 \arrow{d}{\Res_{\pr_1^* r}} \arrow[bend left=70]{dd}{\pr_{2\,\Gamma}} \\
&\Gamma \arrow{r}{\EE G}\arrow{dr}[swap]{\id_\Gamma} & \Gamma_2 \arrow{d}{\EE\pr_2} \\
&& \Gamma 
\end{tikzcd}
\] 
and
\[
\begin{tikzcd}
\Gamma 
\arrow{r}{\Res_r} \arrow{dr}[swap]{\s_\Gamma}  & \Gamma \arrow{r}{\EE G} \arrow{d}{\EE\s} & \Gamma_2 \arrow{d}{\EE\m=\m_\Gamma} \\
&\Hei M \arrow{r}{\EE\unit}[swap]{=\unit_\Gamma} & \Gamma
\end{tikzcd}
\]
The right inverse axiom \eqref{right inverse} follows similarly, using the map $(\id_\jac,\inv):\jac\to\jac_2$.

This shows that $\Gamma$ with these structure maps is a Lie groupoid.

By naturally of $\EA$ and Lemma \ref{Form rescaling},
\[
\pr_{2\,\Gamma}^*(\EA_\jac\theta) = \Res_{\pr_1^*r}^*[\EE\pr_2^*(\EA_\jac\theta)] = \Res_{\pr_1^*r}^* [\EA_{\jac_2}(\pr_2^*\theta) ]
= \EA_{\jac_2}(\pr_1^*e^{r}\cdot\pr_2^*\theta) \ .
\]
This shows that, 
\begin{align*}
\partial^*_\Gamma (\EA_\jac\theta) &=
\pr_{1\,\Gamma}^*(\EA_\jac\theta) - \m_\Gamma^*(\EA_\jac\theta) + \pr_{2\,\Gamma}^*(\EA_\jac\theta) \\
&= \EA_{\jac_2}(\pr_1^*\theta - \m^*\theta + \pr_1^*e^r\pr_2^*\theta) = 0\ ,
\end{align*}
so $\EA_\jac\theta$ --- and thus $\omega_\Gamma= -d(\EA_\jac\theta)$ --- is multiplicative. By the same proof as in Theorem~\ref{Integration thm}, $\omega_\Gamma $ is nondegenerate, therefore $(\Gamma,\omega_\Gamma)$ is a symplectic groupoid.

Contact structures relate to Jacobi structures just as symplectic structures relate to Poisson structures, and
the Poissonization of a contact structure, $\theta$, is the symplectic structure, $-d(e^s\theta) = -d(\hbar^{-1}\theta)$ (see \cite{c-z} with some adjustment of sign conventions). This means that the symplectic structure on the Poissonization $\Hei_+(\jac,\theta)$ is the same as that on the $\hbar>0$ part of $(\Gamma,\omega_\Gamma)$. Because Poissonization is a functor, this shows that the restriction of $\t_\Gamma$ to $\hbar>0$ is a Poisson map. Flipping signs shows that this is also true for $\hbar<0$, and therefore $\t_\Gamma$ is a Poisson map
\end{proof}
\begin{remark}
The functors $\EE$ and $\Hei$ are closely related. Applying $\Hei$ to $(\jac,\theta)$ gives a Poisson groupoid that is symplectic except at $\hbar=0$, and $(\Pr_\jac,\hbar) : \Gamma \to \Hei\jac$ is a Poisson map. There might be an alternative approach to constructing $(\Gamma,\omega_\Gamma)$ by using the Poisson blow-up construction in \cite{g-l}.
\end{remark}

\section{Relation to van Erp's Parabolic Tangent Groupoid}
\label{Parabolic}
In \cite{erp1}, Erik van Erp studied a generalization of the Atiyah-Singer index theorem for contact manifolds using what he calls the \emph{parabolic tangent groupoid}. This is inspired by Connes' tangent groupoid, but whereas Connes' tangent groupoid can be constructed by a simple explosion, the parabolic tangent groupoid requires a double explosion.

A \emph{contact manifold} $(M,E)$ consists of a smooth manifold $M$ and a maximally nonintegrable tangent distribution $E\subset TM$ of codimension $1$. Because the Lie algebroid of $\Pair M$ is $TM$, this is naturally identified with the normal bundle to $\unit:M\to \Pair M$, so $(\Pair M,M,E)$ is an explosive triple. The \emph{parabolic tangent groupoid} $\mathbb{T}_EM$ is the $\hbar\in[0,1]$ part of $\EE(\Pair M,M,E)$.

Over $\hbar=0$, this is a bundle of Heisenberg groups. Specifically, it is a central extension
\[
0 \to TM/E \to \mathbb{T}_EM\rvert_{\hbar=0} \to E \to 0
\]
determined by the curvature of $E$, where the vector bundle $E$ and its normal line bundle $TM/E$ are  bundles of abelian groups.

This is similar to the groupoid $\Gamma$ that I constructed in Section~\ref{Integration2}.  The $\hbar=0$ part of $\Gamma$ is also a bundle of Heisenberg groups. Specifically, it is the central extension
\[
0 \to \R\times M \to \Gamma\rvert_{\hbar=0} \to T^*M \to 0
\]
determined by $\Pi$. 
\begin{remark}
I am calling these ``Heisenberg'' groups even where $\Pi$ is degenerate.
\end{remark}

The similarity is simply due to the fact that both of these groupoids can be constructed by double explosions.
More generally, consider a Lie groupoid $\G$ over $M$ with Lie algebroid $A$. Any subbundle $E\subset A$ determines an explosive triple, $(\G,M,E)$. By the same proof as Lemma~\ref{Groupoid}, $\EE(\G,M,E)$ is a Lie groupoid over $M\times\R$. By Corollary~\ref{h0}, $\EE(\G,M,E)\rvert_{\hbar=0}$ is naturally identified with the graded vector bundle $E\oplus A/E$, and the source and target maps are both equal to the bundle projection, so this is a bundle of groups. The inhomogeneous rescaling ensures that the subbundle $A/E$ and the quotient $A$ are abelian, so this is a central extension
\beq
\label{extension}
0 \to A/E \to \EE(\G,M,E)\rvert_{\hbar=0} \to A \to 0\ .
\eeq

In particular, if $E$ has corank $1$, then this is necessarily a bundle of Heisenberg groups, as in the cases of  $\mathbb{T}_EM$ and $\Gamma$ described above.

The central extension \eqref{extension} is determined by a cocycle. This is most easily computed by considering the Lie algebroid. The Lie algebra of $\EE(\G,M,E)\rvert_{\hbar=0}$ is the graded vector bundle associated to the filtration $0\subset E\subset A$. For two sections $X,Y\in\Gamma(M,E)$, the bracket of $(X,0)$ and $(Y,0)$ is $(0,c(X,Y))$, where $c(X,Y)\in\Gamma(M,A/E)$ is the class of $[X,Y]_A$ modulo $E$. It is simple to check that this defines a bundle map $c:\Wedge^2 E\to A/E$. 

This $c$ is an $A/E$-valued $2$-cocycle on the bundle of abelian Lie algebras $E$. Because these are Abelian, $c$ is also the equivalent cocycle for the bundle of abelian groups $A/E$. In this way, it determines the extension~\eqref{extension}.

This reproduces the definition of the curvature of a contact distribution and suggests possible generalizations.
\end{appendix}

\subsection*{Acknowledgements}
I am very grateful to Chenchang Zhu for pointing out the nonintegrability of the example in Section~\ref{Irrational}; this was immediately obvious to her, but not to me. I would also like to than Ian McIntosh for several useful discussions.

\end{document}